\documentclass{amsart}
\usepackage{amssymb}
\usepackage{graphicx}
\usepackage{enumerate}
\usepackage{multirow}
\usepackage{amsmath,color}
\usepackage{hyperref}
\usepackage{url}

\newtheorem{theorem}{Theorem}[section]

\newtheorem{proposition}[theorem]{Proposition}

\newtheorem{lemma}[theorem]{Lemma}

\newcommand{\G}{\Gamma}
\newcommand{\w}{\omega}

\def\nul{\mathop{\rm nullity}\nolimits}
\def\dist{\mathop{\rm dist }\nolimits}

\begin{document}
\title{Partially metric association schemes with a multiplicity three}

\author{Edwin R. van Dam}
\address{Department of Econometrics and O.R., Tilburg University,
	 the Netherlands}
\email{Edwin.vanDam@uvt.nl}
\author{Jack H. Koolen}
\address{School of Mathematical Sciences,
University of Science and Technology of China and Wu Wen-Tsun Key Laboratory of Mathematics, Chinese Academy of Sciences,
Anhui, P.R. China}
\email{koolen@ustc.edu.cn}
\author{Jongyook Park}
\address{School of Computational Sciences,
Korea Institute for Advanced Study, Seoul, Republic of Korea}
\email{jongyook@hanmail.net}

\subjclass[2010]{05E30, 05C50}

\keywords{association scheme, $2$-walk-regular graph, small multiplicity, distance-regular graph, cover of the cube}

\begin{abstract}
An association scheme is called partially metric if it has a connected relation whose distance-two relation is also a relation of the scheme.
In this paper we determine the symmetric partially metric association schemes with a multiplicity three. Besides the association schemes related to regular complete $4$-partite graphs, we obtain the association schemes related to the Platonic solids, the bipartite double scheme of the dodecahedron, and three association schemes that are related to well-known $2$-arc-transitive covers of the cube: the M\"{o}bius-Kantor graph, the Nauru graph, and the Foster graph F048A. In order to obtain this result, we also determine the symmetric association schemes with a multiplicity three and a connected relation with valency three. Moreover, we construct an infinite family of cubic arc-transitive $2$-walk-regular graphs with an eigenvalue with multiplicity three that give rise to non-commutative association schemes with a symmetric relation of valency three and an eigenvalue with multiplicity three.
\end{abstract}

\maketitle

\section{Introduction}

Bannai and Bannai \cite{BB} showed that the association scheme of the complete graph on four vertices is the only primitive symmetric association scheme with a multiplicity equal to three. They also posed the problem of determining all such imprimitive symmetric association schemes. As many product constructions can give rise to such schemes with a multiplicity three, we suggest and solve a more restricted problem. Indeed, we will determine the symmetric partially metric association schemes with a multiplicity three, where an association scheme is called partially metric if it has a connected relation whose distance-two relation is also a relation of the scheme. Besides the association schemes related to regular complete $4$-partite graphs, we obtain the association schemes related to the Platonic solids, the bipartite double scheme of the dodecahedron, and three association schemes that are related to well-known $2$-arc-transitive covers of the cube: the M\"{o}bius-Kantor graph, the Nauru graph, and the Foster graph F048A. In order to obtain this classification, we also determine the symmetric association schemes with a multiplicity three and a connected relation with valency three and build on work by C\'amara and the authors on $2$-walk-regular graphs \cite{CDKP}.
We furthermore construct an infinite family of cubic arc-transitive $2$-walk-regular graphs with an eigenvalue with multiplicity three that give rise to non-commutative association schemes with a symmetric relation of valency three and an eigenvalue with multiplicity three. The latter indicates that the considered problem is completely different for non-symmetric association schemes and that it may be difficult to classify the cubic $2$-walk-regular graphs with a multiplicity three.

Related work has been done by Yamazaki \cite{Y98}, who showed that if a symmetric association scheme has a connected relation with valency three, then this relation is bipartite or distance-regular. Hirasaka \cite{H00} classified the primitive commutative association schemes with a non-symmetric relation of valency three. Distance-regular graphs with a small multiplicity have also been classified; those with multiplicity three are the graphs of the five Platonic solids and the regular complete $4$-partite graphs. For this and several other results on multiplicities of distance-regular graphs, we refer to \cite[\S~14]{DaKoTa}. See also the expository paper by Bannai \cite{B} on among others the classification problem of association schemes.

This paper is organized as follows: after this introduction, we give definitions and our main tools in Section~\ref{sec:definitions}. In particular, we will use a generalization of Godsil's multiplicity bound \cite[Thm.~1.1]{g88} (Section~\ref{sec:partiallymetricschemes}), a generalization of the concept of a light tail introduced by Juri\v{s}i\'{c}, Terwilliger, and \v{Z}itnik \cite{JTZ10} (Section~\ref{sec:lighttail}), and a lemma by Yamazaki \cite{Y98} (Section~\ref{Yamazaki}). In Section~\ref{sec:uniqueness}, we describe the relevant association schemes and show uniqueness or non-existence of the schemes that occur in Section~\ref{sec:schemesk=m=3} in the proof of the classification result of association schemes with a valency three and a multiplicity three. In Section~\ref{sec:2pmschemes}, we obtain the final classification result of partially metric association schemes with a multiplicity three. Finally, in Section~\ref{sec:cubecovers}, we construct an infinite family of cubic arc-transitive $2$-walk-regular graphs with an eigenvalue with multiplicity three that give rise to non-commutative association schemes.

\section{Definitions and tools}\label{sec:definitions}

In this section we shall introduce notation, concepts, and useful tools that we shall use in the remainder of the paper.

\subsection{Graphs}\label{sec:graphs}
Let $\Gamma$ be a (simple and undirected) graph with vertex set $V$. The {\em distance} $\dist(x,y)$ between two vertices $x,y\in V$ is the length of a shortest
path connecting $x$ and $y$. The maximum distance between
two vertices in $\Gamma$ is the {\em diameter} $D$. We use $\Gamma_i(x)$ for the set of vertices at distance
$i$ from $x$ and write, for the sake of simplicity, $\Gamma(x):=\Gamma_1(x)$. The {\em degree} of $x$ is the number
$|\Gamma(x)|$ of vertices adjacent to it. A graph is {\em regular} with {\em valency} $k$ if the degree of each of its
vertices is $k$.

For a graph $\Gamma$ with diameter $D$, the {\em distance-$i$ graph} $\Gamma_i$ of $\Gamma$ $(1\leq i\leq D)$
is the graph whose vertices are those of $\Gamma$ and whose edges are the pairs of vertices at mutual distance $i$ in
$\Gamma$. In particular, $\Gamma_1=\Gamma$. The {\em distance-$i$ matrix} $B_i$ of $\G$ is the matrix whose rows and columns are indexed by the vertices of $\Gamma$ and the $(x, y)$-entry is $1$ whenever $\dist(x,y)=i$ and $0$ otherwise\footnote{Note that we do not use the notation $A_i$ for the distance-$i$ matrix in order to avoid confusion with the relation matrix of an association scheme; see Section~\ref{sec:associationschemes}}. The {\em adjacency matrix} $A$ of $\G$ equals $B_1$ and the {\em eigenvalues} of the graph $\Gamma$ are those of $A$.  The {\em multiplicity} of an eigenvalue $\theta$ of $\G$ is denoted by $m(\theta)$. Let $\theta_0>\theta_1>\cdots>\theta_r$ be the distinct eigenvalues of $\G$. Then the {\em minimal graph idempotent for $\theta_j$} is defined by $F_j :=\prod_{i\neq j}\frac{A-\theta_i I}{\theta_j-\theta_i}$, i.e., this is the matrix representing the projection onto the eigenspace for $\theta_j$.
The spectral decomposition theorem leads immediately to
\begin{equation}\label{eq: walk regular}
A^{\ell}=\sum_{j=0}^r \theta_j^{\ell}F_j
\end{equation} for every integer $\ell \geq 0$.

\subsection{Walk-regularity}\label{sec:walkregularity}
A connected graph is {\em $t$-walk-regular} if the number of walks of every given length $\ell$ between two vertices $x,y\in V$ only depends on the distance between them, provided that $\dist(x,y)\leq t$ (where it is implicitly assumed that the diameter of the graph is at least $t$).
From \eqref{eq: walk regular}, we obtain that a connected graph is $t$-walk-regular if and only if for every minimal graph idempotent the $(x,y)$-entry
only depends on $\dist(x,y)$, provided that the latter is at most $t$ (see Dalf\'o, Fiol, and Garriga
\cite{DFG09}). In other words, for a fixed minimal graph idempotent $F$ for $\theta$, there exist constants $\alpha_i:=\alpha_i(\theta)$, for $0\leq i\leq t$, such that $B_i\circ F=\alpha_i B_i$, where $\circ$ is the entrywise product.

Given a vertex $x$ in a graph $\Gamma$ and vertex $y$ at distance $i$ from $x$, we consider the numbers
$a_i(x,y)=|\Gamma(y)\cap \Gamma_i(x)|$, $b_i(x,y)=|\Gamma(y)\cap \Gamma_{i+1}(x)|$, and
$c_i(x,y)=|\Gamma(y)\cap\Gamma_{i-1}(x)|$. A connected graph $\Gamma$ with diameter $D$ is {\em distance-regular} if these
parameters do not depend on $x$ and $y$, but only on $i$, for $0\leq i\leq D$. If this is the case then these numbers
are denoted simply by $a_i$, $b_i$, and $c_i$, for $0\leq i\leq D$, and they are called the {\em intersection numbers}
of $\Gamma$. Also, if a connected graph $\Gamma$ is $t$-walk-regular, then the intersection numbers of $\G$ are well-defined for $0\leq
i\leq t$ (see Dalf\'o et al.~\cite[Prop.~3.15]{DvDFGG11}).

\subsection{Association schemes}\label{sec:associationschemes}
Let $X$ be a finite set, say with $n$ elements. An {\em association scheme} with rank $d+1$ on $X$ is a pair $(X,\mathcal{R})$ such that
\begin{enumerate}[(i)]
  \item $\mathcal{R}=\{R_0,R_1,\cdots,R_d\}$ is a partition of $X\times X$,
  \item $R_0:=\{(x,x)\mid x\in X\}$,
  \item for each $i$ $(0 \leq i \leq d)$  $R_i = R_i^{\top}$, i.e., if $(x,y)\in R_i$ then $(y,x)\in R_i$,
  \item there are numbers $p^h_{ij}$ --- the {\em intersection numbers} of $(X,\mathcal{R})$ --- for $0\leq i,j,h\leq d$, such that for every pair $(x,y)\in R_h$ the number of $z\in X$ with $(x,z)\in R_i$ and $(z,y)\in  R_j$ equals $p^h_{ij}$.
\end{enumerate}

In the literature, more general definitions of association schemes are available. We will use these also in Section~\ref{sec:cubecovers}. In particular, we will refer to them as non-symmetric association schemes when not all relations are symmetric (in this case (iii) is replaced by $R_i = R_{i'}^{\top}$ for some $i'$). A non-symmetric association scheme can even be non-commutative in the sense that $p^h_{ij} \neq p^h_{ji}$ for some $h,i,j$.
Association schemes in this broader sense are generalizations of so-called ``Schurian schemes'' that arise naturally from the action of a finite transitive group on $X$; the orbitals (the orbits on $X \times X$) of such a group action form the relations of a (possibly non-commutative or non-symmetric) association scheme.

From now on, we will however assume that association schemes are symmetric (as in the above definition), unless we specify explicitly that it is non-symmetric or non-commutative.

The elements $R_i$ $(0\leq i\leq d)$ of $\mathcal{R}$ are called the {\em relations} of $(X,\mathcal{R})$. For each $i>0$, the relation $R_i$ can be interpreted as a graph $\G$ with vertex set $X$ if we call two vertices $x$ and $y$ adjacent whenever $(x,y) \in R_i$. We call $\G$ the {\em scheme graph} of $R_i$, that is regular with valency $k_i:=p^0_{ii}$. The corresponding adjacency matrix $A_i$ is called the {\em relation matrix} of $R_i$, for $i>0$, and we let $A_0=I$ be the relation matrix of $R_0$. It is easy to see that the conditions (i)-(iv) are equivalent to conditions (i)'-(iv)' on the relation matrices:
\begin{enumerate}[(i)']
  \item $\displaystyle\sum^d_{i=0}A_i=J$, where $J$ is the all-one matrix,
  \item $A_0=I$, where $I$ is the identity matrix,
  \item $(A_i)^{\top}=A_i$ for all $i\in\{0,1\cdots,d\}$,
  \item $A_iA_j=\displaystyle\sum^d_{h=0}p^h_{ij}A_h$.
\end{enumerate}
The {\em Bose-Mesner algebra} $\mathcal{M}$ of $(X,\mathcal{R})$ is the matrix algebra generated by $\{A_i\mid i=0,\ldots,d\}$. From (iv)' we see that $\{A_i\mid i=0,\ldots,d\}$ is a basis of $\mathcal{M}$ and hence $\mathcal{M}$ is $(d+1)$-dimensional. Note that the Bose-Mesner algebra is closed under both ordinary multiplication and entrywise multiplication $\circ$. From (iii)' and (iv)', it follows that the relation matrices commute, and hence all the matrices in $\mathcal{M}$ are simultaneously diagonalizable. It follows that $\mathcal{M}$ has a basis of {\em minimal scheme idempotents} $E_0,E_1,\cdots,E_d$, which we can order such that $nE_0$ is the all-ones matrix $J$. The rank of $E_j$ is denoted by $m_j$ and is called the {\em multiplicity} of $E_j$, for $0\leq j\leq d$.

Now the Bose-Mesner algebra $\mathcal{M}$ has two bases and we can express each basis in terms of the other. Define constants $P_{ji}$ and $Q_{ij}$ $(0\leq i,j\leq d)$ by
\begin{equation}\label{eq: P and Q}
A_i=\sum^d_{j=0}P_{ji}E_j~{\rm and}~E_j=\frac{1}{n}\sum^d_{i=0}Q_{ij}A_i.
\end{equation}
Note that $m_j={\rm rk} E_j={\rm tr} E_j=Q_{0j}$. From \eqref{eq: P and Q}, we have
\begin{equation}\label{eq: eigenvalue}
A_iE_j=P_{ji}E_j,
\end{equation}
hence the numbers $P_{ji}$ are called the {\em eigenvalues} of $(X,\mathcal{R})$.
In this paper, we shall mainly focus on the eigenvalues of the scheme graph of $R_1$. In this case we call $P_{j1}$ $(0\leq j\leq d)$ the {\em corresponding eigenvalue} on $E_j$ and it is denoted by $\theta_j$, i.e., $A_1E_j=\theta_jE_j$. Note that these eigenvalues $\theta_j$ need not be distinct, for example in the Johnson scheme $J(7,3)$ defined on the triples of a $7$-set, the relation defined by ``intersecting in $1$ point'' has this property (because it is strongly regular in an association scheme with rank $4$).

Since the minimal scheme idempotents form a basis of $\mathcal{M}$, we have
\begin{equation}\label{eq: krein}
E_i\circ E_j=\frac{1}{n}\sum^d_{h=0}q^h_{ij}E_h
\end{equation}
for certain real numbers $q^h_{ij}$ ($0\leq h,i,j\leq d$) that are called {\em Krein parameters}. The Krein parameters are nonnegative and $q^0_{ij}=\delta_{ij}m_j$, where $\delta_{ij}$ is $1$ whenever $i=j$ and $0$ otherwise.
From \eqref{eq: P and Q}, we also have
\begin{equation}\label{eq: cos}
 E_j\circ A_i=\frac{Q_{ij}}{n}A_i.
 \end{equation}
It follows that $(E_j)_{xx}=\frac{Q_{0j}}{n}=\frac{m_j}{n}$ for all $x\in X$. For $(x,y)\in R_i$, let $\w_{xy}=\w_{xy}(j)=\frac{(E_j)_{xy}}{(E_j)_{xx}}=\frac{Q_{ij}/n}{m_j/n}=\frac{Q_{ij}}{m_j}$. We call these numbers $\w_i=\w_i(j)=\frac{Q_{ij}}{m_j}$ the {\em cosines} corresponding to $E_j$, and note that $\w_0=1$. From \eqref{eq: eigenvalue} and \eqref{eq: cos}, it follows that if $(x,y)\in R_h$, then
\begin{equation}\label{eq: eigenvector}
P_{ji}\w_{h}=P_{ji}\w_{xy}=\displaystyle\sum_{z\in R_i(x)}\w_{zy}=\sum_{\ell=1}^d p^h_{i\ell}\w_{\ell},
\end{equation}
where (here and elsewhere) $R_i(x)=\{z\mid (x,z)\in R_i\}$.

From a standard property of the entries of $Q$, see \cite[Lemma~2.2.1.(iv)]{bcn89}, we obtain that
\begin{equation}\label{eq: multcos}
m_j\sum_{i=0}^dk_i\w_i^2=n.
\end{equation}

For more background on association schemes, see \cite{BI}, \cite[Ch.~2]{bcn89}, and \cite{MT}.

\subsection{Partially metric association schemes and Godsil's bound}\label{sec:partiallymetricschemes}
An association scheme $(X,\mathcal{R})$ with rank $d+1$ is called {\em $t$-partially metric (with respect to the connected relation $R$)} if --- possibly after reordering of the relations --- $A_i$ is a polynomial of degree $i$ in $A$ for $i =1,2,\ldots, t$, where $A$ is the relation matrix of $R=R_1$ (where implicitly it is assumed that $t \leq d$). This is equivalent to $R_i$ being the distance-$i$ graph of the scheme graph of $R$ for $i \leq t$. Note that the distance-$i$ graph $\G_i$ of a scheme graph $\G$ is always a union of relations $R_j$. The scheme $(X,\mathcal{R})$ is called {\em metric} if it is $d$-partially metric; in this case $R$ is a distance-regular graph.
For the sake of readability, we will assume in the remainder of the paper that for a partially $t$-metric scheme, the relations are ordered according to distance, up to distance $t$ (as in the above definition), unless specified differently. We note that a $t$-partially metric scheme is clearly also $s$-partially metric for $s \leq t$. Every association scheme with at least one connected relation is $1$-partially metric; we therefore call an association scheme {\em partially metric} if it is at least $2$-partially metric. To ensure that every metric association scheme is also partially metric, we also say that an association scheme with rank $2$ (where there is no distance-$2$ relation) is partially metric. We finally note that the concept of $t$-partially metric can be extended to non-symmetric schemes with respect to a symmetric relation. Such $t$-partially metric (possibly non-symmetric) schemes would arise naturally from $t$-arc-transitive graphs, for example; see also Section~\ref{sec:cubecovers}.

If the association scheme $(X,\mathcal{R})$ is $t$-partially metric, then the {\em corresponding scheme graph} $\G$ of $R_1$ is called a {\em $t$-partially metric scheme graph}. This scheme graph is $t$-partially distance-regular in the sense of \cite{DvDFGG11}, and even stronger, it is $t$-walk-regular. Thus, the intersection numbers of $\G$ are well-defined for $0\leq i\leq t$. In this case we have $a_i=p^i_{1i}, b_i=p^i_{1,i+1}$ and $c_i=p^i_{1,i-1}$ for $0\leq i\leq t-1$, and $a_t=p^t_{1t}$, $c_t=p^t_{1,t-1}$ and  $b_t=b_0-a_t-c_t$, where $b_0=k=p^0_{11}$ is the valency of $\Gamma$. An illustrating example of a $3$-partially metric scheme graph is given by the so-called flag graph of the $11$-point biplane; see Figure 1 in \cite{CDKP} or \cite{DvDFGG11} for the corresponding ``relation-distribution diagram''. Such a diagram is similar as the distance-distribution diagram of a distance-regular graph. The {\em relation-distribution diagram} of an association scheme with respect to a scheme graph $R_1$ has a ``bubble'' for each relation $R_i$, inside of which we depict $k_i$, and we connect the bubble of $R_i$ by an edge to the bubble of $R_j$ if $p^i_{1j}>0$, and depict this intersection number on top of the edge; see for example Figure~\ref{fig:dodecahedron}.

From \eqref{eq: eigenvector}, we now obtain that
\begin{align*}
\theta &=b_0\w_1\\
\theta \w_h&=c_h \w_{h-1}+a_h \w_h+b_h\w_{h+1} \qquad (1\leq h\leq t-1),
\end{align*}
where $\w_h$ $(0\leq h\leq d)$ are the cosines corresponding to a minimal scheme idempotent $E$ for corresponding eigenvalue $\theta$.
It follows in particular that if $t \geq 2$, then
\begin{equation}\label{eq: smallcosines}
\w_0=1,~ \w_1=\theta/k,~ \w_2=\frac{\theta^2-a_1\theta-k}{kb_1}.
\end{equation}

As an immediate consequence of \cite[Thm.~4.3]{CDKP}, we find the following generalization of Godsil's bound \cite[Thm.~1.1]{g88}.

\begin{theorem}\label{thm: godsil}
Let $(X,\mathcal{R})$ be a partially metric association scheme and assume that the corresponding scheme graph $\G$ has valency $k\geq3$. Let $E$ be a minimal scheme idempotent of $(X,\mathcal{R})$ with multiplicity $m$ for corresponding eigenvalue $\theta\neq\pm k$. If $\G$ is not complete multipartite, then $k\leq\frac{(m+2)(m-1)}{2}$.
\end{theorem}
This result implies that if $k\geq3$ then $m \geq 3$. For $k=2$, we only have the polygons and they have multiplicity $2$ for all minimal scheme idempotents except those for corresponding eigenvalue $\pm 2$. If $\G$ is complete multipartite, then $d=2$. In this case, it follows that if $k \geq 3$, then multiplicity $2$ only occurs for the complete tripartite graphs, and multiplicity $1$ only occurs for eigenvalue $\pm k$ of the complete bipartite graphs. Multiplicity $3$ occurs only for the complete $4$-partite graphs and the complete tripartite cocktail party graph, also known as the octahedron.

\subsection{Product schemes and the bipartite double}\label{sec:bipartitedouble}

Let $(X, {\mathcal R})$ be an association scheme with rank $d+1$ and relation matrices $A_i$ for $i=0,1,\dots,d$, and let $(X', {\mathcal R}')$ be an association scheme with rank $d'+1$ with relation matrices $A_j'$ for $j=0,1,\dots,d'$. The {\em direct product} of $(X, {\mathcal R})$ and $(X', {\mathcal R}')$ is the association scheme with relation matrices $A_i \otimes A_j'$ for $i=0,1,\dots,d$ and $j=0,1,\dots,d'$. It is easy to see that the minimal idempotents of this direct product scheme are also all possible Kronecker products of the minimal idempotents of $(X, {\mathcal R})$ and $(X', {\mathcal R}')$; see also \cite[Chapter~3]{bailey}. Starting from an association scheme $(X, {\mathcal R})$ with a multiplicity three, one can construct other association schemes with a multiplicity three by taking the direct product of $(X, {\mathcal R})$ with any other scheme. Also other kinds of product constructions for association schemes are possible, giving rise to many association schemes with a multiplicity three, and suggesting that classifying all association schemes with a multiplicity three may be impossible. Likewise, multiplicity two may be too hard, although in this case our result in \cite[Prop.~6.5]{CDKP} should be useful.

The {\em bipartite double scheme} $BD(X,\mathcal{R})$ of $(X,{\mathcal R})$ is the direct product of $(X,{\mathcal R})$ and the rank two association scheme on two vertices. In this way, every minimal idempotent of $(X,{\mathcal R})$ with multiplicity $m$ corresponds to two minimal idempotents of $BD(X,\mathcal{R})$ with multiplicity $m$. For a connected graph $\Gamma$ with vertex set $V$, the {\em bipartite double} of $\Gamma$ is the graph whose vertices are the symbols $x^{+},x^{-}~ (x\in V)$ and where $x^+$ is adjacent to $y^-$ if and only of $x$ is adjacent to $y$ in $\G$. If $\Gamma$ is the scheme graph of a relation $R$ in $(X, {\mathcal R})$, then the bipartite double of $\G$ is a scheme graph in the bipartite double scheme $BD(X,\mathcal{R})$.

If $(X, {\mathcal R})$ is $t$-partially metric with corresponding scheme graph $\G$ having odd-girth at least $2t+1$, then the bipartite double of $(X, {\mathcal R})$ is also $t$-partially metric. This result follows from the arguments given in the proof of the analogous result for $t$-walk-regular graphs in \cite[Prop.~3.1]{CDKP}.

\subsection{Quotient schemes and covers}\label{sec:covers}

An association scheme is called {\em imprimitive} if a non-trivial union of some of the relations is an equivalence relation. In this case, there is a subscheme on each of the equivalence classes, and a quotient scheme on the set of equivalence classes. The original scheme is called a cover of the quotient scheme. The intersection numbers and Krein parameters of the subschemes and the quotient scheme follow from those of the original scheme. Like all direct product schemes, the bipartite double scheme $BD(X,\mathcal{R})$ is an example of an imprimitive association scheme; it is a double cover of $(X,\mathcal{R})$. For details, we refer the reader to \cite[\S~2.9]{BI}, \cite[\S~2.4]{bcn89}, or \cite{DMM}.

A particular way to construct covers of graphs is by using voltage graphs. Let $\G=(V,E)$ be a graph and let $(G,+)$ be a group. Let $\vec{E}$ be the set of arcs of $\G$ (for every edge $\{x,y\}$, there are two opposite arcs: $(x,y)$ and $(y,x)$). A map $\alpha: \vec{E} \rightarrow G$ such that $\alpha(x,y)=-\alpha(y,x)$ for every edge $\{x,y\}$ is called a {\em voltage assignment}, and $(V,E,\alpha)$ is called a {\em voltage graph}. The {\em derived graph} $\G'$ of this voltage graph is a cover of $\G$; it has vertex set $V \times G$, and if $\{x,y\}$ is an edge in $\G$, then $\G'$ has edges $\{(x,g),(y,g+\alpha(x,y))\}$ for every $g\in G$. Every double cover is the derived graph of a voltage graph with group $\mathbb{Z}_2$. In this case, the situation is simpler, and we can put voltages on the edges instead of the arcs. For example, the bipartite double can be obtained by putting voltage $1$ on every edge.

\subsection{A light tail}\label{sec:lighttail}

Let $(X,\mathcal{R})$ be a partially metric association scheme and let $A$ be the relation matrix of $R_1$. A minimal scheme idempotent $E:=E_j$  for corresponding eigenvalue $\theta$ is called a {\em light tail} if the matrix $F:=\displaystyle\sum_{h\neq0}q^h_{jj}E_h$ is nonzero and $AF=\eta F$ for some real number $\eta$. Thus, if $q^h_{jj} \neq 0$, then the corresponding eigenvalue on $E_h$ is equal to $\eta$, for all $h\neq0$. Because $R_1$ is connected, this also implies that $\eta \neq k$. We call $F$ the {\em associated matrix }for $E$ and $\eta$ the {\em corresponding eigenvalue} on $F$. We call the light tail {\em degenerate} if $\theta=\eta$ and {\em non-degenerate} otherwise. This generalizes the concept of light tails in distance-regular graphs that was introduced by Juri\v{s}i\'{c}, Terwilliger, and \v{Z}itnik \cite{JTZ10}. Note that $F=nE \circ E -mE_0$ by \eqref{eq: krein}, where $m=m_i$ is the rank of $E$, which implies that $F_{xx}=\frac1n m(m-1)$. Because $F$ is positive semidefinite, it follows that $F=0$ if and only if $m=1$. By Theorem~\ref{thm: godsil} and the remarks thereafter this is equivalent to $\theta=\pm k$, where $k$ is the valency of $R_1$.
Let us now define $\widetilde{F}:=\frac{n}{m(m-1)}F$, so that $\widetilde{F}_{xx}=1$. Because $\widetilde{F}$ is in the Bose-Mesner algebra of $(X,\mathcal{R})$, there are $\rho_0,\rho_1,\ldots,\rho_d$ such that $\widetilde{F}\circ A_i=\rho_i A_i$ for all $i=0,1,\dots,d$. Similar as for minimal scheme idempotents, we call these numbers the {\em cosines} corresponding to $F$, and we let $\rho_{xy}=\rho_i$ for $(x,y)\in R_i$. Similar as \eqref{eq: eigenvector}, the following now holds for $(x,y)\in R_h$:
\begin{equation*}\label{eq: eigenvector2}
\eta \rho_h=\displaystyle\sum_{z\in R_1(x)}\rho_{zy}=\sum_{\ell=1}^dp^h_{1\ell}\rho_{\ell}.
\end{equation*}
In particular, this implies that
$\rho_0=1, \rho_1=\eta/k,$ and $\rho_2=\frac{\eta^2-a_1\eta-k}{kb_1}.$
It moreover follows from the equation $F=nE \circ E -mE_0$ that
\begin{equation}\label{eq: cosinesrhoomega}
(m-1)\rho_i=m\w_i^2-1,
\end{equation}
where $\w_i$ are the cosines corresponding to $E$, for $i=0,1,\dots,d$. Working out this equation for $i=1$ gives that
\begin{equation}\label{eq: etatheta}
(m-1)\eta=\frac{m}{k}\theta^2-k.
\end{equation}

Our generalization of light tails is motivated by the characterization of the case of equality in the following result on the multiplicities of minimal scheme idempotents. For distance-regular graphs, this bound was derived by Juri\v{s}i\'{c}, Terwilliger, and \v{Z}itnik \cite{JTZ10}, and their proof can be followed almost completely.

\begin{theorem}\label{thm: light tail 2}  Let $(X,\mathcal{R})$ be a partially metric association scheme with rank $d+1\geq3$, and assume that the corresponding scheme graph $\G$ has valency $k\geq3$. Let $E$ be a minimal scheme idempotent with multiplicity $m$ for corresponding eigenvalue $\theta\neq\pm k$. Then
\begin{equation}\label{eq: light tail 1}
m\geq k-\frac{k(\theta+1)^2a_1(a_1+1)}{((a_1+1)\theta+k)^2+ka_1b_1},
\end{equation}
with equality if and only if $E$ is a light tail.
\end{theorem}

\begin{proof}
We give a sketch of the proof of the first part, as most details are the same as in the case of distance-regular graphs; see \cite[Thm~3.2 and 4.1]{JTZ10}. Let $j$ be such that $E=E_j$. Then the bound \eqref{eq: light tail 1} follows from applying Cauchy-Schwarz to
$$\mathbf{v_0}=\left[\sqrt{q_{jj}^1m_1},\ldots,\sqrt{q_{jj}^dm_d}\right]~~{\rm and}~~\mathbf{v_1}=\left[\theta_1\sqrt{q_{jj}^1m_1},\ldots,\theta_d\sqrt{q_{jj}^dm_d}\right].$$
The bound is tight if and only if $\mathbf{v_0}$ and $\mathbf{v_1}$ are linearly dependent, which is the case if and only if $\theta_h$ is the same for all $h \neq 0$ such that $q^h_{jj}\neq 0$, in other words, if and only if $E_j$ is a light tail.
\end{proof}

\subsection{Yamazaki's lemma}\label{Yamazaki}
The following result was shown by Yamazaki \cite{Y98} and is analogous to the result that a cubic $1$-walk-regular graph is $2$-walk-regular \cite{CDKP}. For convenience and because the terminology in \cite{Y98} is different, we give a proof of this result.

\begin{lemma} {\em (}cf. \cite[Lemma 2.4]{Y98}{\em )} \label{prop: cubic 1-walk} Let $(X,\mathcal{R})$ be an association scheme with rank $d+1\geq3$. If there exists a connected relation $R\in\mathcal{R}$ with valency three, then $(X,\mathcal{R})$ is partially metric with respect to $R$.
\end{lemma}
\begin{proof}
Let $\G$ be the scheme graph of $R=:R_1$. Because the rank of the scheme is at least $3$, $\G$ is not the complete graph on $4$ vertices, and so $a_1=p^1_{11}=0$. If $\G_2$ is not a relation of the scheme, then it must be the union of two relations, $R_2$ and $R_3$ say, and then $p^1_{12}=p^1_{13}=1$. Now let $x$ be a vertex of $\Gamma$ and let $y_1,y_2,y_3$ be the three neighbors of $x$. Clearly these three are mutually at distance $2$. Without loss of generality, we may assume that $(y_1,y_2)\in R_2$ and $(y_1,y_3)\in R_3$ because $p^1_{12}=p^1_{13}=1$. But then $(y_2,y_3)$ should be contained in both $R_2$ and $R_3$, which is a contradiction. Thus, $(X,\mathcal{R})$ is partially metric with respect to $R$.
\end{proof}

The final lemma, which we shall call Yamazaki's lemma, is also from \cite{Y98}. Again, we give a proof for convenience and because of the different terminology in \cite{Y98}. The result is depicted in Figure~\ref{fig: cherry2}.

\begin{figure}
\begin{center}
\includegraphics[scale=0.7]{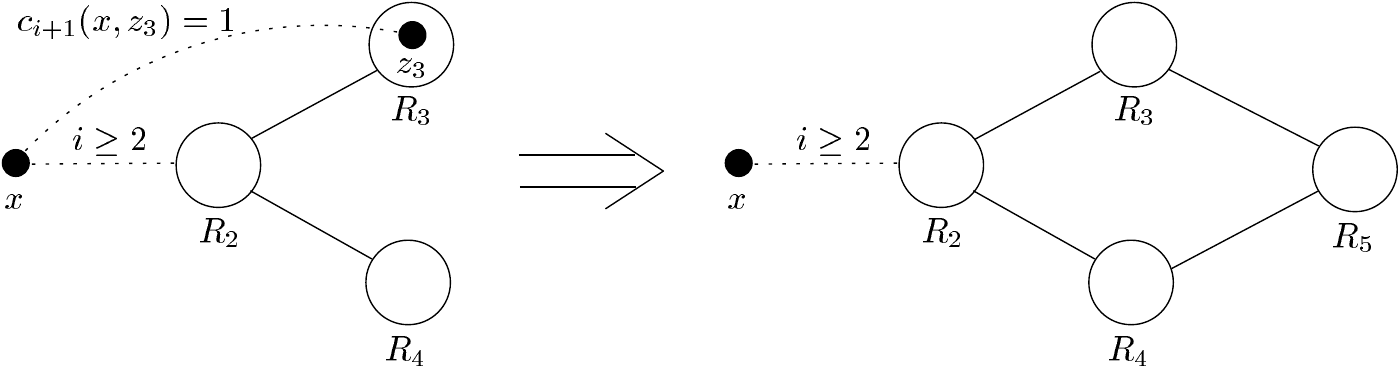}
\caption{A graphical interpretation of Yamazaki's lemma~\ref{lem: cherry}}
\label{fig: cherry2}
\end{center}
\end{figure}

\begin{lemma}{\em (}cf. \cite[Lemma 2.8]{Y98}{\em )} \label{lem: cherry}
Let $(X,\mathcal{R})$ be an association scheme with rank $d+1 \geq 4$ and a connected scheme graph $\G$ of $R_1\in \mathcal{R}$. Let $x,z$ be vertices such that $(x,z)\in R_2$ and $\dist(x,z)=i \geq 2$. Assume that there exist two distinct neighbors $z_3,z_4$ of $z$ and two distinct relations $R_3,R_4\in \mathcal{R}$ such that $(x,z_3)\in R_3$, $(x,z_4)\in R_4$, $\dist(x,z_3)=\dist(x,z_4)=i+1$, and $c_{i+1}(x,z_3)=1$. Then there exists a relation $R_5\in \mathcal{R}$ such that $p^1_{35}\neq0$, $p^1_{45}\neq0$, and $R_5 \cap \G_i = \emptyset$.
\end{lemma}

\begin{proof}
Because the association scheme is symmetric, there exist a neighbor $v_3$ of $x$ such that $(z,v_3)\in R_3$, $\dist(z,v_3)=i+1$, and $c_{i+1}(z,v_3)=1$. Let $R_5$ be the relation containing $(v_3,z_4)$. See Figure~\ref{fig: cherry1} for a picture of this configuration. Then $p^1_{35}\neq0$ as $(z,z_4)\in R_1$, $(z,v_3)\in R_3$, and $(v_3,z_4)\in R_5$, and similarly $p^1_{45}\neq0$ as $(x,v_3)\in R_1$, $(x,z_4)\in R_4$, and $(v_3,z_4)\in R_5$.

In order to show that $R_5 \cap \G_i = \emptyset$, it suffices to show that $\dist(v_3,z_4)\neq i$.
From $\dist(z,v_3)=i+1$ and $c_{i+1}(z,v_3)=1$, it is clear that $\Gamma(v_3)\cap\Gamma_i(z)=\{x\}$. By symmetry and because $\G_i$ is a union of relations of $\mathcal{R}$, there exists a unique vertex $y$ such that $\G(z)\cap\G_i(v_3)=\{y\}$, and it follows that $\dist(x,y)=i-1$. But $\dist(x,z_4)=i+1$, hence $\dist(v_3,z_4)\neq i$.
\end{proof}

\begin{figure}
\begin{center}
\includegraphics[scale=0.7]{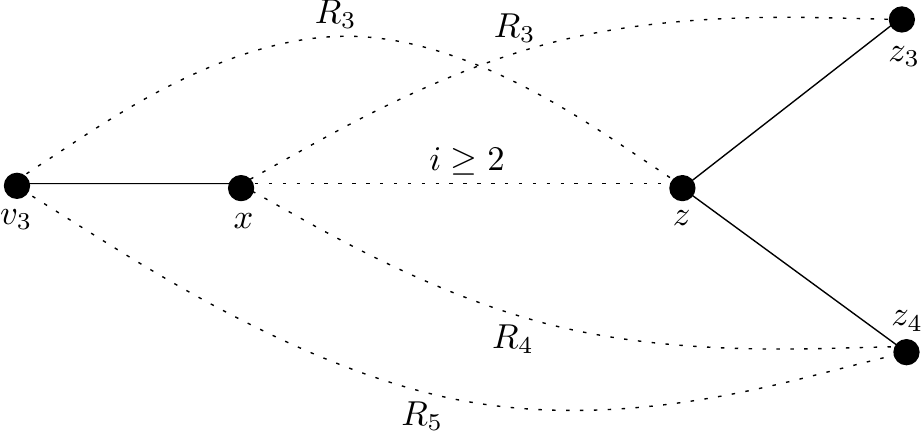}
\caption{The configuration of vertices in the proof of Yamazaki's lemma~\ref{lem: cherry}}
\label{fig: cherry1}
\end{center}
\end{figure}

\section{Uniqueness and non-existence of the relevant association schemes}\label{sec:uniqueness}

In this section, we will discuss some of the relevant association schemes having a multiplicity three that occur in the proof of the classification result in Section~\ref{sec:schemesk=m=3}.

\subsection{The dodecahedron}\label{ssec:dod} The dodecahedron graph is a distance-regular graph with spectrum $\{3^1,\sqrt{5}^3,1^5,0^4,-2^4,-\sqrt{5}^3\}$. Thus, both the corresponding metric association scheme and its bipartite double scheme have minimal scheme idempotents with a multiplicity three. Note however that the bipartite double graph does not have an eigenvalue with multiplicity three; its spectrum is $\{3^1,\sqrt{5}^6,2^4,1^5,0^8,\break -1^5,-2^4,-\sqrt{5}^6,-3^1\}$. The relation-distribution diagram of the bipartite double scheme is given in Figure~\ref{fig:dodecahedron}, where we also included the cosines for eigenvalue $\sqrt{5}$ that we obtained in the proof of Theorem~\ref{thm: classification}. We note that the bipartite double graph is also the scheme graph of a $3$-partially metric fusion scheme of the bipartite double scheme. This scheme can be obtained by fusing three times a pair of relations (i.e., $R_3\cup R_4,R_5\cup R_8$, and $R_{11}\cup R_{12}$; see Figure~\ref{fig:dodecahedron}). However, also three pairs of idempotents are ``fused'', in particular two pairs of idempotents with multiplicity three, leaving no multiplicity three in this fusion scheme.

\begin{figure}[h!]
\begin{center}
\includegraphics[scale=0.65]{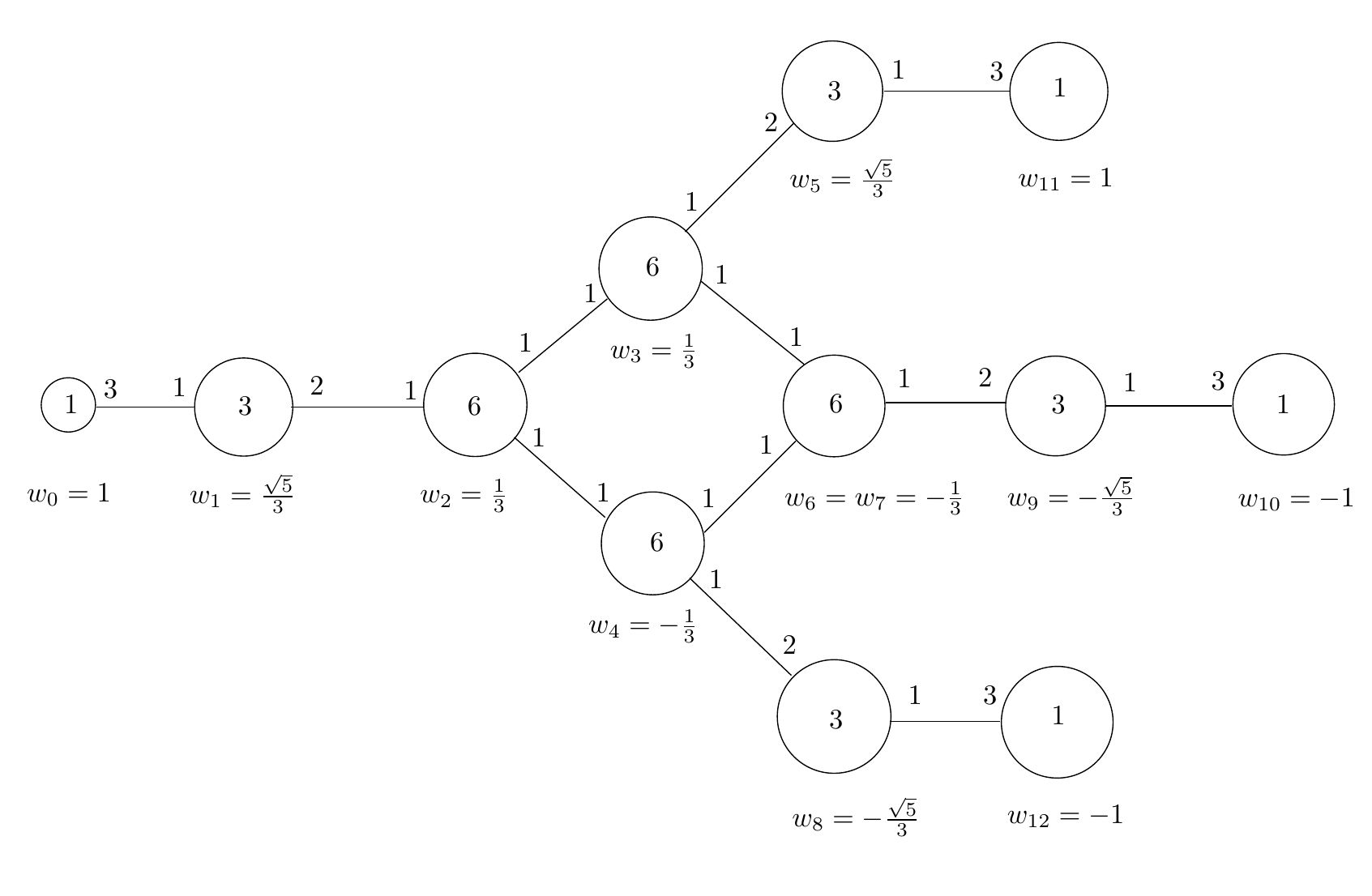}
\caption{Relation-distribution diagram of the bipartite double of the dodecahedron}
\label{fig:dodecahedron}
\end{center}
\end{figure}

\begin{proposition}\label{prop:dodecahedron} The bipartite double of the association scheme of the dodecahedron graph is the unique association scheme with scheme graph having relation-distribution diagram as in Figure~\ref{fig:dodecahedron}.
\end{proposition}

\begin{proof} Because $R_{11}$ has valency $1$, the relation $R_0\cup R_{11}$ is clearly an equivalence relation. If we take the quotient scheme with respect to this equivalence relation, we obtain an association scheme for which the scheme graph obtained from $R_1\cup R_5$ is distance-regular with valency three and distance-distribution diagram as that of the dodecahedron; this follows from Figure~\ref{fig:dodecahedron}. Because the dodecahedron and the corresponding association scheme is determined by its intersection numbers, this quotient scheme is indeed the metric association scheme of the dodecahedron. But then (the scheme graph) $R_1$ is a bipartite double cover of the dodecahedron, and hence it must be the bipartite double graph of the dodecahedron. Moreover, the association scheme is therefore the bipartite double scheme of the association scheme of the dodecahedron.
\end{proof}

\subsection{The M\"{o}bius-Kantor graph}\label{ssec:mk}The M\"{o}bius-Kantor graph is the unique double cover of the cube without $4$-cycles \cite[p.~267]{bcn89}. It is isomorphic to the generalized Petersen graph $GP(8,3)$ and has spectrum $\{3^1,\sqrt{3}^4,1^3,-1^3,-\sqrt{3}^4,-3^1\}$. It is $2$-arc-transitive and also known as the Foster graph F016A \cite{RCMD}. It generates an association scheme with scheme graph having relation-distribution diagram as in Figure~\ref{fig:mk}, where also the cosines for eigenvalue $1$ are included; these cosines follow from the relation distribution diagram using \eqref{eq: eigenvector}.

\begin{figure}[h!]
\begin{center}
\includegraphics[scale=0.65]{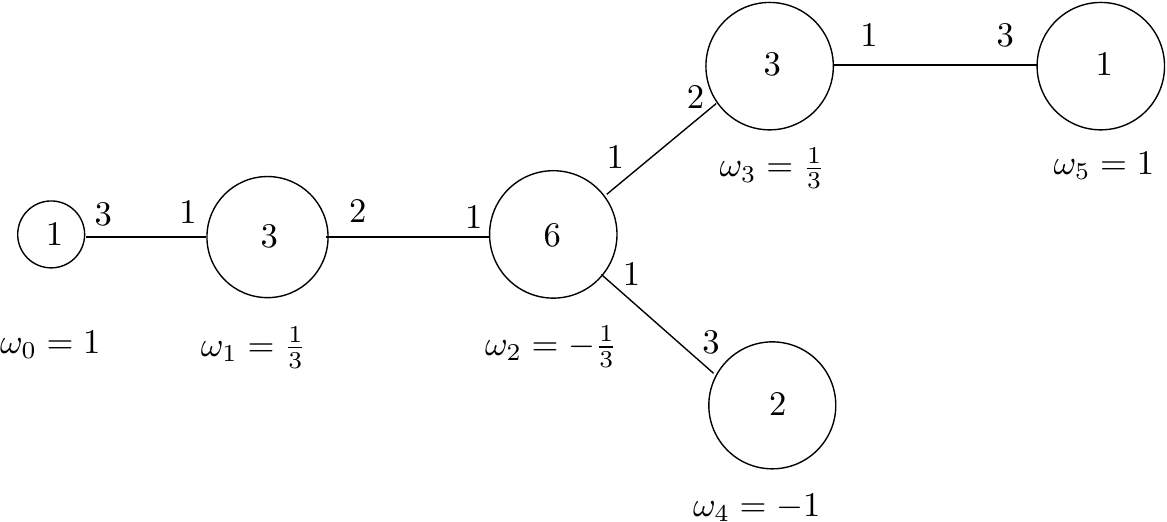}
\caption{Relation-distribution diagram of the M\"{o}bius-Kantor graph}
\label{fig:mk}
\end{center}
\end{figure}

\begin{proposition}\label{prop:mk} The association scheme of the M\"{o}bius-Kantor graph is the unique association scheme with scheme graph having relation-distribution diagram as in Figure~\ref{fig:mk}.
\end{proposition}

\begin{proof} Fix a vertex. Then it is easy to see that there is just one way (up to isomorphism) to build the graph with the given relation-distribution diagram from the perspective of the fixed vertex and using that the graph has no $4$-cycles. The obtained graph is the M\"{o}bius-Kantor graph; clearly, the other relations of the association scheme follow from this.
\end{proof}

\subsection{The Nauru graph}\label{ssec:nauru}The Nauru graph is a triple cover of the cube. It is isomorphic to the generalized Petersen graph $GP(12,5)$ and has spectrum $\{3^1, 2^6, 1^3,\linebreak 0^4, -1^3, -2^6, -3^1\}$. It is $2$-arc-transitive and also known as the Foster graph F024A \cite{RCMD}. It generates an association scheme with scheme graph having relation-\break distribution diagram as in Figure~\ref{fig:nauru}, where also the cosines for eigenvalue $1$ are included; again these cosines follow from the relation distribution diagram using \eqref{eq: eigenvector}.

\begin{figure}[h!]
\begin{center}
\includegraphics[scale=0.65]{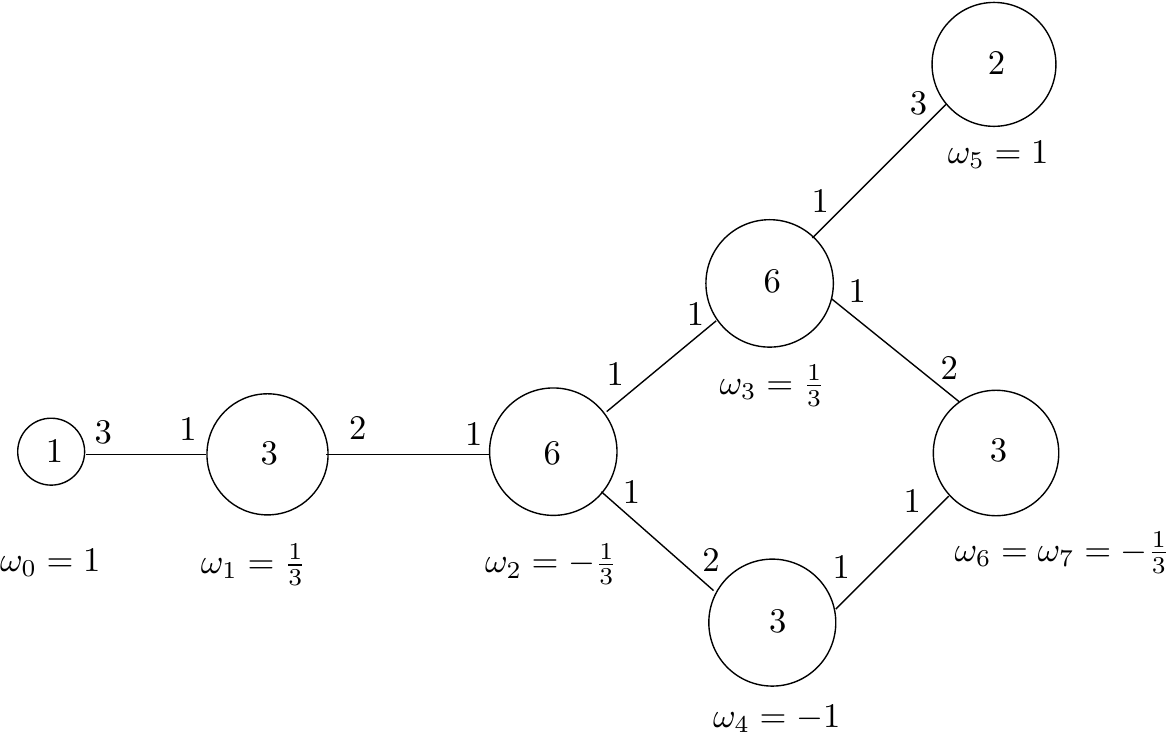}
\caption{Relation-distribution diagram of the Nauru graph}
\label{fig:nauru}
\end{center}
\end{figure}

\begin{proposition}\label{prop:nauru} The association scheme of the Nauru graph is the unique association scheme with scheme graph having relation-distribution diagram as in Figure~\ref{fig:nauru}.
\end{proposition}

\begin{proof} Fix a vertex $x$. If one ignores the six edges between $R_2(x)$ and $R_3(x)$, then up to isomorphism, one can build the graph with the given relation-distribution diagram (seen from the perspective of $x$) in a unique way (up to equivalence), using that there are no $4$-cycles. It is easy to show that $R_0 \cup R_5$ is an equivalence relation. Now fix one of the vertices $y\in R_5(x)$. Then (again, up to equivalence) there is a unique way to determine the sets $R_i(y)$ for all $i$ (i.e., there are two equivalent ways to determine $R_6(y)$ and $R_2(y)$; the rest is determined). Also observe (by considering the edges through $x$) that every edge is in precisely two $6$-cycles that share no other edges. If we apply this to the edges between $R_1(x)$ and $R_2(x)$, use the relation distribution with respect to $y$ (i.e., the sets $R_i(y)$), and that there are no $4$-cycles, then the remaining six edges of the scheme graph follow uniquely. In particular, observe that each edge between $R_1(x)$ and $R_2(x)$ is in one $6$-cycle with vertices from $\cup_{i=0,1,2,4}R_i(x)$ and in one $6$-cycle with vertices from $\cup_{i=1,2,3,6}R_i(x)$, and the latter determines the edges between $R_2(x)$ and $R_3(x)$. The obtained graph is the Nauru graph; and again, the other relations of the association scheme follow from this.
\end{proof}

\subsection{The Foster graph F048A}\label{ssec:foster48}The Foster graph F048A is a $6$-cover of the cube, a $3$-cover of the M\"{o}bius-Kantor graph, and a $2$-cover of the Nauru graph. It is isomorphic to the generalized Petersen graph $GP(24,5)$ and has spectrum $\{3^1, \sqrt{6}^4, 2^6, \sqrt{3}^4, 1^3, 0^{12}, -1^3, -\sqrt{3}^4, -2^6, -\sqrt{6}^4, -3^1\}$. It is $2$-arc-transitive \cite{RCMD} and generates an association scheme with scheme graph having relation-distribution diagram as in Figure~\ref{fig:foster48}, where as before, the cosines for eigenvalue $1$ are included; once more these cosines follow from the relation distribution diagram using \eqref{eq: eigenvector}. In this association scheme, the eigenvalue $0$ has two minimal scheme idempotents; these have multiplicities $4$ and $8$.

\begin{figure}[h!]
\begin{center}
\includegraphics[scale=0.65]{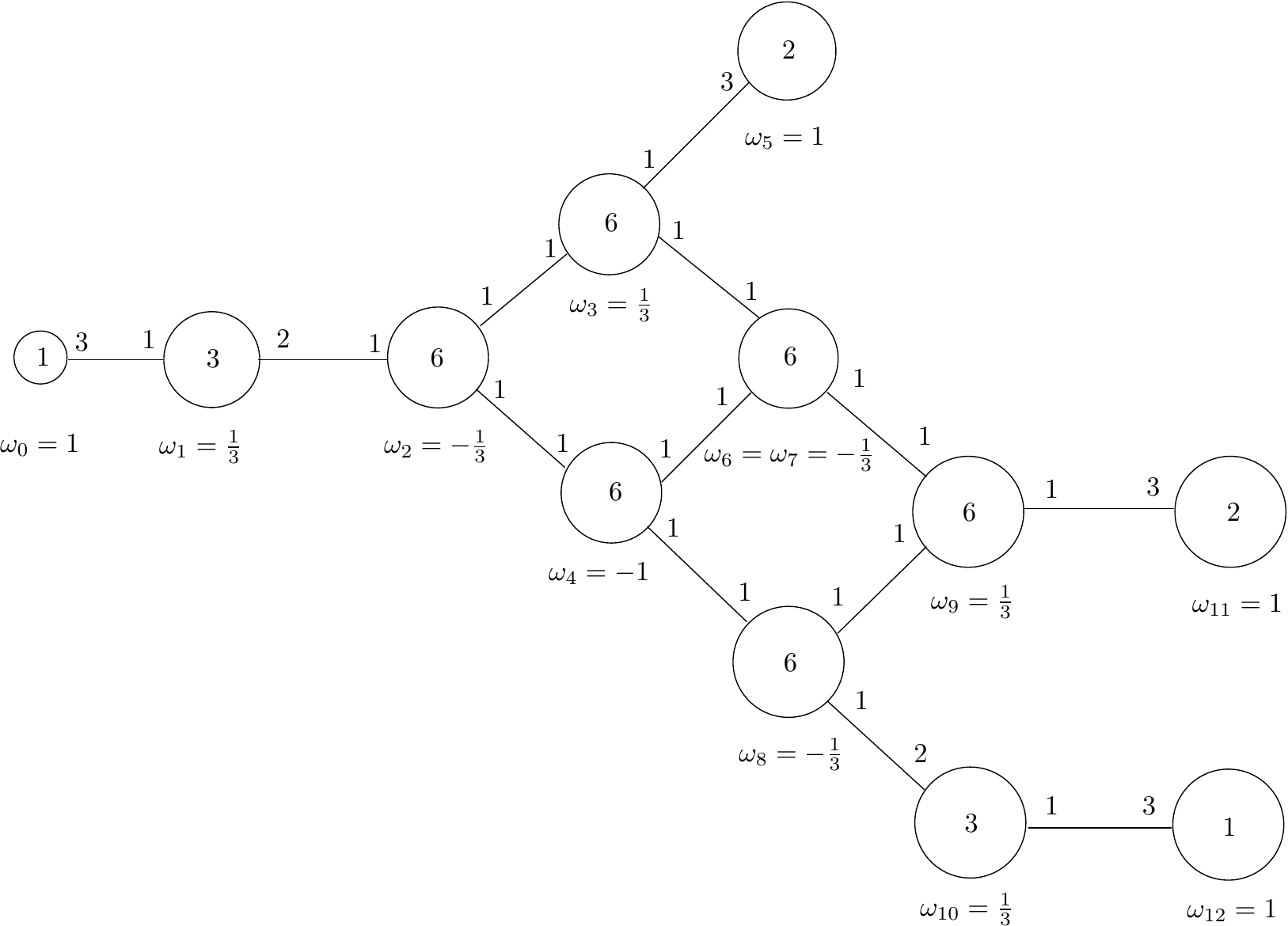}
\caption{Relation-distribution diagram of the Foster graph F048A}
\label{fig:foster48}
\end{center}
\end{figure}

\begin{proposition}\label{prop:foster48} The association scheme of the Foster graph {\em F048A} is the unique association scheme with scheme graph having relation-distribution diagram as in Figure~\ref{fig:foster48}.
\end{proposition}

\begin{proof} Let $\G$ be the scheme graph of an association scheme $(X,\mathcal{R})$, with relation-distribution diagram as in Figure~\ref{fig:foster48}. Because the valency of $R_{12}$ is $1$, it follows that this association scheme has a quotient scheme, $(V,\mathcal{S})$ say, with a scheme graph having relation-distribution diagram as in Figure~\ref{fig:nauru} (and let us number the relations of $(V,\mathcal{S})$ as in that figure). By Proposition~\ref{prop:nauru}, this quotient scheme must therefore be the association scheme of the Nauru graph. Thus, $\G$ is a $2$-cover of the Nauru graph, and it can be constructed from a voltage graph with group $\mathbb{Z}_2$; see Section~\ref{sec:covers}. We will next show that there is essentially one way to do this. In order to do this, we will use the description of the Nauru graph as a generalized Petersen graph $GP(12,5)$. This graph has vertices $i$ and $i^*$ with $i \in \mathbb{Z}_{12}$, where $i$ has neighbors $i-1,i^*$, and $i+1$, whereas $i^*$ has neighbors $(i-5)^*,i$, and $(i+5)^*$. Without loss of generality, we may put voltage $0$ on the edges of a spanning tree of $GP(12,5)$. The spanning tree we will use has edges $\{i,i-1\}$ for $i \neq 0$ and $\{i,i^*\}$ for all $i$.

Next, we will focus on relation $S_5$ of the quotient scheme. It splits into relations $R_5$ and $R_{11}$ in the cover scheme $(X,\mathcal{R})$, where we note that $R_5$ is among the distance-$4$ relations, whereas $R_{11}$ is among the distance-$6$ relations. Because the three walks of length $4$ between two vertices $x$ and $y$ with $(x,y)\in S_5$ should give rise to three walks of length $4$ between two vertices $x'=(x,g)$ and $y'=(y,h)$ with $(x',y')\in R_5$ (for some $g,h \in \mathbb{Z}_2$), it follows that these three walks should have the same voltage. Here the voltage of a walk is the sum of voltages over the edges in the walk. In particular, $(6,10)\in S_5$, with one of the walks between $6$ and $10$ having voltage $0$ (being part of the spanning tree), which implies that also the other two walks should have voltage $0$. This implies that $(6^*,11^*)$ and $(5^*,10^*)$ have voltage $0$. Similarly, $(4^*,9^*),(3^*,8^*),(2^*,7^*),(1^*,6^*),$ and $(0^*,5^*)$ have voltage $0$.

Finally, we observe that because in the cover graph $\G$ there are no $6$-cycles, the voltage of a $6$-cycle in the Nauru graph should be $1$ (where similar as before, the voltage of a cycle is the sum of voltages of its edges). This observation determines the voltages of all remaining edges, as one can easily see. In particular, note that every $6$-cycle consists of consecutive adjacent vertices $i,i+1,(i+1)^*,(i+6)^*,(i-1)^*,i-1,i$ for some $i \in \mathbb{Z}_{12}$, from which it follows that the voltages of the edges $(0,11),(7^*,0^*),(8^*,1^*),(9^*,2^*),(10^*,3^*)$, and $(11^*,4^*)$ must be $1$. The obtained derived graph is the generalized Petersen graph $GP(24,5)$, that is, the Foster graph F048A. Also here, the other relations of the association scheme follow from this.
\end{proof}

We note that this result is confirmed by considering the subscheme on one of the bipartite halves and the computational classification of association schemes with $24$ vertices by Hanaki and Miyamoto \cite{MH,MHsite}, and by observing that such a subscheme determines the entire scheme because the girth of the scheme graph is $8$.

\subsection{A putative fission scheme for the Coxeter graph}\label{ssec:coxeter}

The Coxeter graph is the unique distance-regular graph with intersection array $\{3,2,2,1;1,1,1,2\}$ \cite[Thm.~12.3.1]{bcn89}. It is also known as the Foster graph F028A \cite{RCMD}. In the following, we will show that it is impossible to fission the distance-$3$ relation in the corresponding association scheme. The bipartite double of such a putative fission scheme occurs as one of the cases in the proof of Theorem~\ref{thm: classification}.

\begin{figure}[h!]
\begin{center}
\includegraphics[scale=0.60]{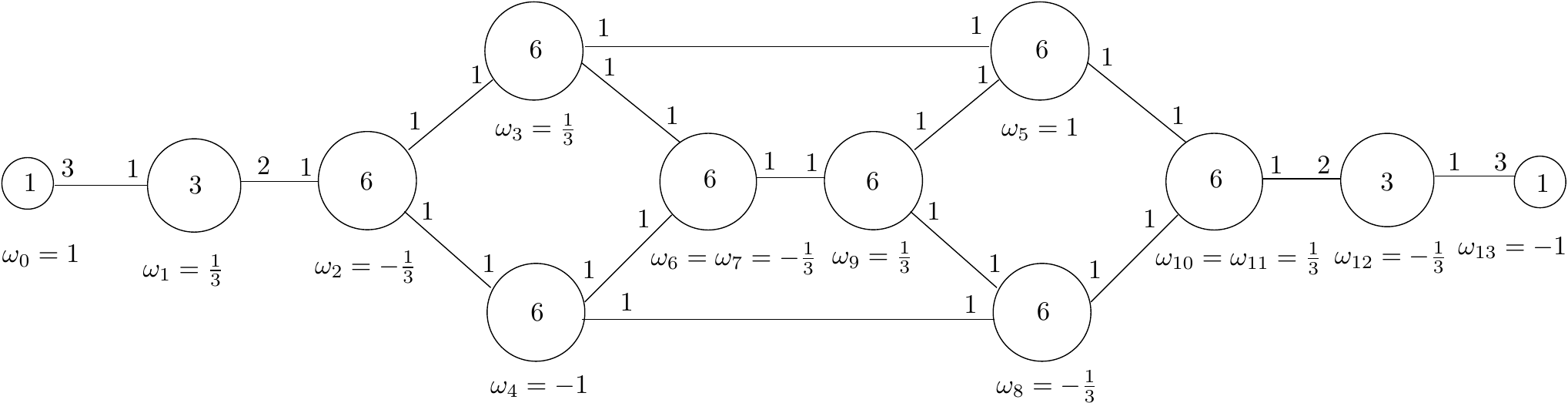}
\caption{Relation-distribution diagram of the bipartite double of a putative fission scheme of the Coxeter graph}
\label{fig:doublecoxeter}
\end{center}
\end{figure}

\begin{figure}[h!]
\begin{center}
\includegraphics[scale=0.65]{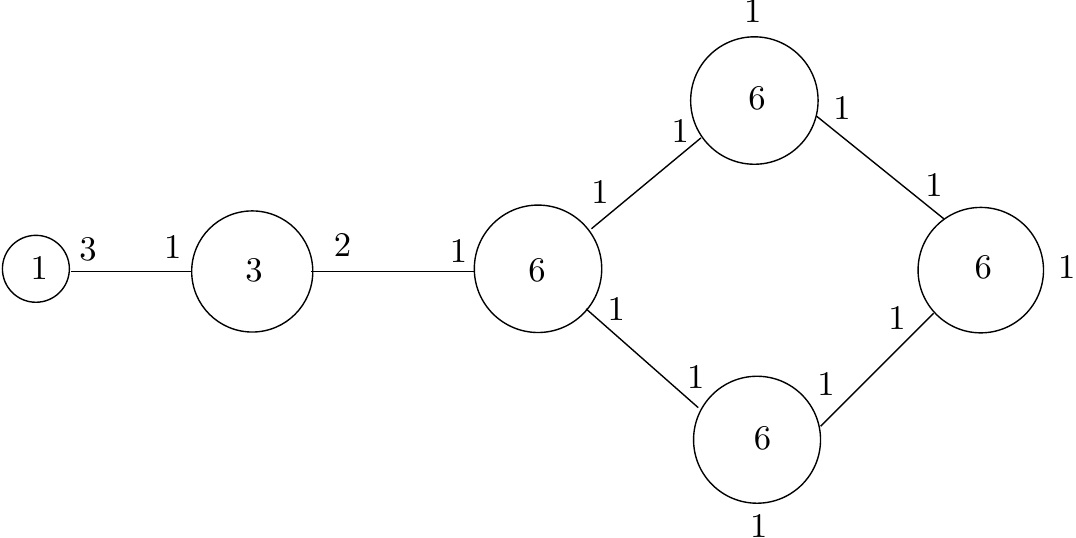}
\caption{Relation-distribution diagram of a putative fission scheme of the Coxeter graph}
\label{fig:fissioncoxeter}
\end{center}
\end{figure}

\begin{proposition}\label{prop:noncoxeter} There is no association scheme with scheme graph having\break relation-distribution diagram as in Figure~\ref{fig:doublecoxeter} or Figure~\ref{fig:fissioncoxeter}.
\end{proposition}

\begin{proof} It is easy to see that a scheme of Figure~\ref{fig:doublecoxeter} must be a double cover of a scheme of Figure~\ref{fig:fissioncoxeter}. If we fuse the relations at distance $3$ in the latter, we obtain a scheme of a distance-regular graph with intersection array $\{3,2,2,1;1,1,1,2\}$. It is known that there is a unique such distance-regular graph, the Coxeter graph. Thus, the scheme graph is the Coxeter graph. Now fix a vertex in the Coxeter graph. The induced graph on the set of vertices at distance $2$ and $3$ is the disjoint union of two $9$-cycles. It is easy to see that this makes it impossible to partition the vertices at distance $3$ into two sets of size $6$ with the intersection numbers as in Figure~\ref{fig:fissioncoxeter}. Thus, no such association schemes exist.
\end{proof}

\subsection{A putative $3$-cover of the scheme of the M\"{o}bius-Kantor graph}\label{ssec:3covermk}
Another case that appears in the proof of Theorem~\ref{thm: classification} is that of a putative $3$-cover of the M\"{o}bius-Kantor graph, with relation-distribution diagram as in Figure~\ref{fig:3covermk}. Here we will show that the related association scheme does not exist. From the intersection matrix $L_1$, that is defined by $(L_1)_{ij}=p^i_{1j}$, and which follows from the relation-distribution diagram, one can compute the eigenmatrix $P$ (see \eqref{eq: P and Q}) of the association scheme because in this case $L_1$ has no repeated eigenvalues. From this, all other intersection numbers, multiplicities, and Krein parameters can be computed; see for example \cite[p.~46]{bcn89}. It turns out that some intersection numbers, such as $p^6_{88}$, and several Krein parameters are negative. Moreover, some multiplicities are not integral. Our proof will avoid these computations though.

\begin{figure}[h!]
\begin{center}
\includegraphics[scale=0.65]{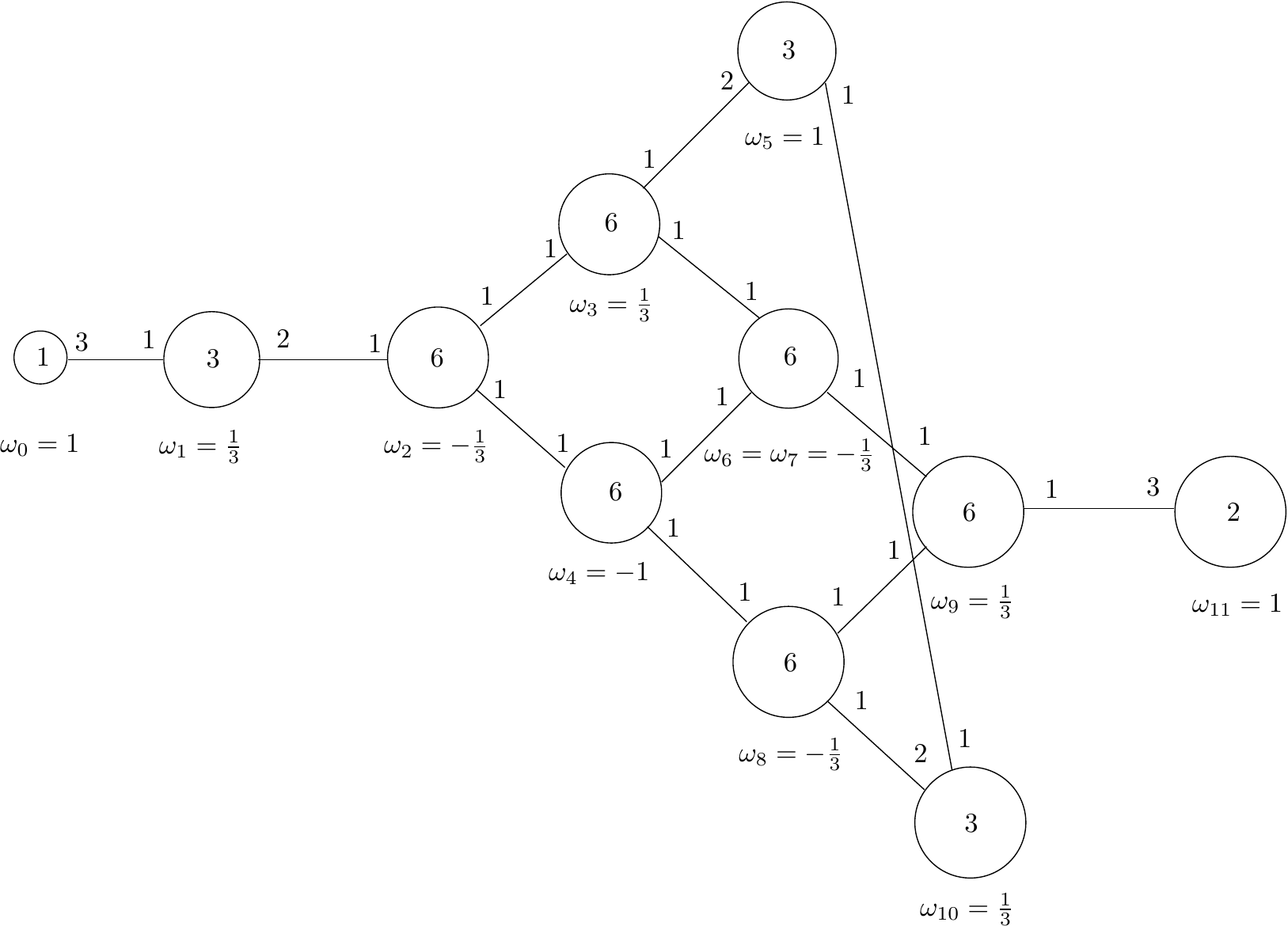}
\caption{Relation-distribution diagram for a putative $3$-cover of the M\"{o}bius-Kantor graph}
\label{fig:3covermk}
\end{center}
\end{figure}

\begin{proposition}\label{prop:non3covermk} There is no association scheme with scheme graph having\break relation-distribution diagram as in Figure~\ref{fig:3covermk}.
\end{proposition}

\begin{proof} From the relation-distribution diagram and \eqref{eq: eigenvector}, it follows easily that the only possible cosine sequence for eigenvalue $0$ is $(\psi_0,\dots,\psi_6,\psi_8,\dots,\psi_{11})=(1,0,-\frac12,0,0,\break 0,\frac12,0,0,0,-\frac12)$. Alternatively, this is the only normalized eigenvector of $L_1$ for eigenvalue $0$.
By \eqref{eq: multcos}, the corresponding multiplicity equals $48/4.5$, which is not integral, so such an association scheme cannot exist.
\end{proof}

We note that this result is confirmed by considering the putative subscheme on one of the bipartite halves and the computational classification of association schemes with $24$ vertices by Hanaki and Miyamoto \cite{MH,MHsite}.

\section{Association schemes with a valency and multiplicity three}\label{sec:schemesk=m=3}

In this section we shall determine the association schemes having a connected relation with valency three and a minimal scheme idempotent with multiplicity three. In order to find this classification, we first need another lemma on a certain configuration of vertices and the corresponding cosines.

\begin{lemma} \label{lemma: multi 3} Let $(X,\mathcal{R})$ be a partially metric association scheme with a connected scheme graph $\G$ with valency $k=3$ and $a_1=0$. Let $E$ be a minimal scheme idempotent with multiplicity three and let $\theta$ be the corresponding eigenvalue of $\G$ on $E$.
Let $u_1$ and $u_2$ be two adjacent vertices in $\G$, let $v_1, v_2$ be the other two neighbors of $u_1$, and $v_3, v_4$ be the other two neighbors of $u_2$. Fix another vertex $x$, and let $\psi_i=\w_{xu_i}$ $(i=1,2)$ and $\phi_i=\w_{xv_i}$ $(i=1,2,3,4)$ be the respective cosines corresponding to $E$. Then $$\phi_3,\phi_4=\frac12(\theta \psi_2 -\psi_1 \pm(\phi_1-\phi_2)).$$
\end{lemma}

\begin{proof} Because $k=3$ and the multiplicity $m$ equals $3$, it follows that $\theta\neq \pm3$.
In order to calculate the cosines corresponding to $E$, we will use the following well-known approach. Because $E$ has rank $3$, it can be written as $E=UU^{\top}$, where $U$ is an $n \times 3$ matrix with columns forming an orthonormal basis of the eigenspace of $E$ for its eigenvalue $1$, with $n$ being the number of vertices of $\Gamma$. For every vertex $u$ of $\Gamma$ we denote by $\hat{u}$ the row of $U$ that corresponds to $u$, normalized to length $1$. Now the inner product $\langle\hat{u},\hat{v}\rangle$ is equal to $\frac n3 E_{uv}=\w_{uv}$.

Now, let $L$ be the orthogonal complement (in $\mathbb{R}^3$) of the subspace spanned by $\hat{u_1}$ and $\hat{u_2}$. The latter two vectors are linearly independent because $\theta \neq \pm 3$ and $u_1$ is adjacent to $u_2$, and hence $L$ is $1$-dimensional. From the equations $\langle\hat{u_1},\hat{v_1}\rangle=\w_1=\langle\hat{u_1},\hat{v_2}\rangle$ and
$\langle\hat{u_2},\hat{v_1}\rangle=\w_2=\langle\hat{u_2},\hat{v_2}\rangle$, it follows that $\hat{v_1}-\hat{v_2}$ is in $L$.
Similarly, $\hat{v_3}-\hat{v_4}$ is in $L$. It is also easily shown that $\|\hat{v_1}-\hat{v_2}\|^2=2-2\w_2=\|\hat{v_3}-\hat{v_4}\|^2$, and hence it follows that $\hat{v_3}-\hat{v_4}=\pm(\hat{v_1}-\hat{v_2})$. By taking the inner product with $\hat{x}$, it thus follows that $\phi_3-\phi_4=\pm(\phi_1-\phi_2)$.

On the other hand, from $AE=\theta E$ (evaluated at $(u_2,x)$), we find that
$\phi_3+\phi_4=\theta \psi_2-\psi_1$. By combining the two obtained equations we now find the required equation for $\phi_3$ and $\phi_4$.
\end{proof}

By Lemma~\ref{prop: cubic 1-walk}, an association scheme is partially metric if it has a connected relation with valency three.
We will now show that if in addition it has a minimal scheme idempotent with multiplicity three, then the corresponding eigenvalue is $\pm1$ or $\pm\sqrt5$.

\begin{proposition} \label{thm: multi 3} Let $(X,\mathcal{R})$ be an association scheme with rank $d+1\geq3$ and with a connected scheme graph $\G$ with valency $k=3$. Let $E$ be a minimal scheme idempotent with multiplicity three and let $\theta$ be the corresponding eigenvalue of $\G$ on $E$. Then $\theta\in\{\pm 1,\pm\sqrt5\}$. Moreover, if $\theta = \pm \sqrt{5}$, then $c_3=1$.
\end{proposition}

\begin{proof}
Again, because $k=3$ and the multiplicity $m$ equals $3$, it is clear that $\theta\neq \pm3$. By Lemma~\ref{prop: cubic 1-walk}, $(X,\mathcal{R})$ is partially metric with respect to $R_1$, the relation with scheme graph $\G$, and it follows that $a_1=0$. Without loss of generality, we may assume that $\G$ is bipartite. Indeed, if $\Gamma$ is not bipartite, then we can consider the bipartite double of $\Gamma$ which is the scheme graph of $BD(X,\mathcal{R})$, which is also partially metric because the odd-girth of $\G$ is at least $5$ (see Section~\ref{sec:bipartitedouble}), and which has minimal scheme idempotents with multiplicity three for corresponding eigenvalues $\theta$ and $-\theta$. So we assume that $\Gamma$ is bipartite.

We will now first show that $c_2=1$ or $\theta\in\{-1,1\}$.
In order to show this claim, let $xz_1z_2$ be a path of length $2$ in $\G$, i.e., $(x,z_1)\in R_1$, $(z_1,z_2) \in R_1$, and $(x,z_2)\in R_2$. Let $\w_0,\ldots, \w_d$ be the cosines corresponding to $E$. As $(X,\mathcal{R})$ is partially metric, $\w_{xz_1}=\w_{z_1z_2}=\w_1=\frac{1}{3}\theta$ and $\w_{xz_2}=\w_2=\frac{1}{6}(\theta^2-3)$ by \eqref{eq: smallcosines}.

Let $z_3$ and $z_4$ be the two neighbors of $z_2$ different from $z_1$, with $(x,z_3)\in R_3$ and $(x,z_4)\in R_4$ for some relations $R_3,R_4\in \mathcal{R}$. Note that $R_1, R_3$, and $R_4$ are not necessarily distinct. However, by calculating the cosines $\w_3$ and $\w_4$ in terms of $\theta$, we will show that $R_3$ and $R_4$ are distinct from $R_1$ if $\theta\neq\pm1$, which will prove that in this case $c_2(x,z_2)=1$. Indeed, by applying Lemma~\ref{lemma: multi 3} to the adjacent vertices $z_1,z_2$ and their neighbors, we find that $\w_3,\w_4=\frac12(\theta \w_2 -\w_1 \pm(\w_0-\w_2))$. Working this out in terms of $\theta$ gives (without loss of generality) that
\begin{equation}\label{eq: w3}
\w_3=\tfrac 12(\theta \w_2-\w_1 +1-\w_2)=\tfrac{1}{12}(\theta^3-\theta^2-5\theta+9)
\end{equation}
and
\begin{equation}\label{eq: w4}
\w_4=\tfrac12(\theta \w_2-\w_1-1+\w_2)=\tfrac{1}{12}(\theta^3+\theta^2-5\theta-9).
\end{equation}
Now it easily follows that if $\w_3=\w_1$ or $\w_4=\w_1$, then $\theta\in\{-1,1\}$, which shows the claim. Note also that $\w_3\neq \w_4$, and hence $R_3\neq R_4$, because $\theta\neq\pm3$.

Next, let us assume that $c_2=1$. We next claim that also $c_3=1$ or $\theta\in\{-1,1\}$.
In order to prove this claim, we consider the neighbors of $z_3$ and $z_4$. Note that $z_3$ and $z_4$ are at distance $3$ from $x$ as $c_2=1$. Let $z_5$ and $z_6$  be the two neighbors of $z_3$ different from $z_2$, and similarly let $z_7$ and $z_8$ be the two neighbors of $z_4$ different from $z_2$. We assume that $(x,z_i)\in R_{i}$ for some relations $R_{i}$, for $i=5,6,7,8$. Our aim is to show that $R_i \neq R_2$ for $i=5,6,7,8$ if $\theta \neq \pm 1$. In order to do this, we again calculate the corresponding cosines in terms of $\theta$, using Lemma~\ref{lemma: multi 3}. Indeed, this gives that $\w_5,\w_6=\frac12(\theta \w_3 -\w_2 \pm(\w_1-\w_4))$. Together with \eqref{eq: w3} and \eqref{eq: w4}, it follows (without loss of generality) that
\begin{equation}\label{eq: w5}
\w_5=\tfrac12(\theta \w_3-\w_2+\w_1-\w_4)=\tfrac{1}{24}(\theta^4-2\theta^3-8\theta^2+18\theta+15)
\end{equation}
and
\begin{equation*}\label{eq: w6}
\w_6=\tfrac12(\theta \w_3-\w_2-\w_1+\w_4)=\tfrac{1}{24}(\theta^4-6\theta^2-3).
\end{equation*}
Similarly, we obtain that
\begin{equation*}\label{eq: w6-1}
\w_7=\tfrac12(\theta \w_4-\w_2+\w_1-\w_3)=\tfrac{1}{24}(\theta^4-6\theta^2-3)
\end{equation*}
and
\begin{equation*}\label{eq: w5-1}
\w_8=\tfrac12(\theta \w_4-\w_2-\w_1+\w_3)=\tfrac{1}{24}(\theta^4+2\theta^3-8\theta^2-18\theta+15).
\end{equation*}
Again, it easily follows that $\w_i \neq \w_2$ for $i=5,6,7,8$ if $\theta \neq \pm 1$, which indeed shows the claim that $c_3=1$ or $\theta\in\{-1,1\}$.

We now first observe that $E$ is a light tail according to Theorem~\ref{thm: light tail 2} because $a_1=0$ and $m=k=3$, and hence we have equality in \eqref{eq: light tail 1}.
Let $F$ be the associated matrix for $E$ for corresponding eigenvalue $\eta$, i.e., $AF=\eta F$, and let $\rho_0\ldots,\rho_d$ be the cosines corresponding to $F$. Then $\eta= \frac 12 \theta^2-\frac 32$ by \eqref{eq: etatheta} and $\rho_i=\frac 32 \w_i^2-\frac 12$ for $i=0,1,\dots,d$ by \eqref{eq: cosinesrhoomega}.

Assume now that $c_3=1$. We continue with the above configuration of vertices. By Yamazaki's lemma~\ref{lem: cherry} we then know that $p^5_{14}\neq0$ or $p^6_{14}\neq0$.
We will first show that $p^5_{14}=0$. Indeed, assume that $z_5$ has a neighbor $s_4$ such that $(x,s_4)\in R_4$. Besides $z_3$ and $s_4$, $z_5$ has one more neighbor, $z_9$ say, with $(x,z_9)\in R_9$ for some relation $R_9$. From $AE=\theta E$, we obtain that $\w_3+\w_4+\w_9=\theta \w_5$, which implies that
\begin{equation}\label{eq: w9}
\w_9=\tfrac{1}{24}(\theta^5-2\theta^4-12\theta^3+18\theta^2+35\theta).
\end{equation}
From $AF=\eta F$, we obtain that $\rho_3+\rho_4+\rho_9=\eta \rho_5$. By substituting
$\eta= \frac 12 \theta^2-\frac 32$ and $\rho_i=\frac 32 \w_i^2-\frac 12$ for $i=0,1,\dots,d$ into this equation, and then \eqref{eq: w3}, \eqref{eq: w4}, \eqref{eq: w5}, and \eqref{eq: w9}, we obtain that\footnote{We used Mathematica to work this out}
\begin{equation*}\label{eq: case1}
(\theta^2-4\theta-1)(\theta+1)^3(\theta-3)^3(\theta+3)^2=0.
\end{equation*}
Thus, $\theta=2 \pm \sqrt{5}$. However, because these eigenvalues are not integral, as algebraic conjugates they must both be eigenvalues of $\G$. But $2+\sqrt{5}>3$, and so it cannot be an eigenvalue. Hence, by contradiction, $p^5_{14}=0$, and it follows that $p^6_{14}\neq0$.

Finally, we consider the neighbors of $z_6$. Similar as above, we now obtain that
\begin{equation*}\label{eq: case1111}
(\theta^2-5)(\theta-1)^2(\theta+1)^2(\theta-3)^2(\theta+3)^2=0.
\end{equation*}
Thus, $\theta=\pm\sqrt5$.
\end{proof}

We are now ready to classify the association schemes having a connected relation with valency three and a minimal scheme idempotent with multiplicity three.

\begin{theorem} \label{thm: classification} Let $(X,\mathcal{R})$ be an association scheme with rank $d+1$, a connected scheme graph $\G$ with valency three, and a minimal scheme idempotent with multiplicity three. Then one of the following holds:
\begin{enumerate}[{\em (i)}]
  \item $d=1$ and $\G$ is the tetrahedron (the complete graph on $4$ vertices)
  \item $d=3$ and $\G$ is the cube,
  \item $d=5$ and $\G$ is the M\"{o}bius-Kantor graph,
  \item $d=6$ and $\G$ is the Nauru graph,
  \item $d=11$ and $\G$ is the Foster graph {\em F048A},
  \item $d=5$ and $\G$ is the dodecahedron,
  \item $d=11$ and $\G$ is the bipartite double of the dodecahedron.
\end{enumerate}
Moreover, the association scheme $(X,\mathcal{R})$ is uniquely determined by $\Gamma$. In all cases, except {\em (vii)}, this is the association scheme that is generated\footnote{That is, the association scheme of minimal rank that has $\G$ as a scheme graph} by $\G$. In case {\em (vii)}, the association scheme is the bipartite double scheme of the association scheme of case {\em (vi)}.
\end{theorem}

\begin{proof} Let $E=E_j$ be the minimal scheme idempotent with multiplicity $m=3$, for eigenvalue $\theta$.

(1){\em The tetrahedron.}
If the rank of $(X,\mathcal{R})$ is $2$, then $\G$ is the complete graph on $4$ vertices, and we have case (i).

So from now on, we may assume that the rank $d+1$ is at least $3$, and hence Proposition~\ref{thm: multi 3} applies, and $\theta \in \{\pm1,\pm\sqrt{5}\}$. As in the proof of Proposition~\ref{thm: multi 3}, it follows that $(X,\mathcal{R})$ is partially metric with respect to $\G$ and $a_1=0$. Likewise, by considering the bipartite double scheme, we may first restrict ourselves to determining the bipartite graphs $\G$ and corresponding association schemes, but then we also have to check afterwards which association schemes could have $(X,\mathcal{R})$ as their bipartite double.

(2){\em The cube.}
So we assume that $\Gamma$ is bipartite. First, note that if $c_2=3$, then $\G$ must be the complete bipartite graph $K_{3,3}$, but the corresponding scheme does not have an idempotent with multiplicity three. Secondly, if $c_2=2$, then it is easily found that $\G$ must be the cube. The corresponding rank $4$ scheme indeed has two minimal scheme idempotents with multiplicity three (for $\theta=\pm 1$). The only scheme having this scheme as its bipartite double is the scheme of the complete graph on $4$ vertices, and hence we obtain cases (i) and (ii).

For the remaining cases, we may assume that $c_2=1$.

(3){\em Eigenvalue $\sqrt{5}$; the dodecahedron.}
We first consider the case $\theta=\pm \sqrt5$; without loss of generality we assume that $\theta=\sqrt{5}$. For this case, we consider the configuration of vertices and the notation in the proof of Proposition~\ref{thm: multi 3}. We thus find the following cosines: $\w_0=1$, $\w_1=\frac{\sqrt5}{3}$, $\w_2=\w_3=\frac{1}{3}$, $\w_4=-\frac{1}{3}$, $\w_5=\frac{\sqrt5}{3}$, $\w_6=\w_7=-\frac{1}{3}$ and $\w_8=-\frac{\sqrt5}{3}$. It also follows that $c_3=1$, $p^5_{14}=0$, and similarly $p^8_{13}=0$.

We claim now that Figure~\ref{fig:dodecahedron} shows the relation-distribution of $\Gamma$. Indeed, the distribution up to distance $3$ is clear. In order to show the remainder of the distribution, we first observe that because $c_3=1$, it follows that $R_6=R_7$ by Yamazaki's lemma~\ref{lem: cherry}. Because $p^6_{13}=p^3_{16}k_3/k_6=6/k_6=p^4_{16}k_4/k_6=p^6_{14}$ and $p^6_{13}+p^6_{14}\leq 3$, it follows that $p^6_{13}=p^6_{14}=1$. Let $R_9$ be the (remaining) relation , at distance $5$, such that $p^6_{19}=1$. Note that this is not the relation $R_9$ in the proof of Proposition~\ref{thm: multi 3}. Then it follows from \eqref{eq: eigenvector} (with $P_{j1}=\theta=\sqrt{5}$) that $\w_9=\sqrt{5}\w_6-\w_3-\w_4=-\frac{\sqrt5}{3}$.
Next, we will show that $p^5_{13}=2$. Indeed, suppose that $p^5_{13}=1$. Then by Yamazaki's lemma~\ref{lem: cherry}, it follows that $p^9_{15}>0$. However, by applying Lemma~\ref{lemma: multi 3} to the adjacent vertices $z_3$ and $z_5$ and their neighbors, the cosines\footnote{By the cosine of a vertex $z$ we mean the cosine $\w_{xz}$} for the remaining two neighbors of $z_5$ are
$\frac12(\theta \w_5 -\w_3 \pm(\w_2-\w_6))$, which equal $\frac13$ and $1$, and one of these should be $\w_9$, which is a contradiction. These cosines also show that $p^5_{13} \neq 3$, and hence it follows indeed that $p^5_{13}=2$. Similarly, we obtain that $p^8_{14}=2$. We now have the distribution up to distance $4$, and observe that $c_4=2$. The latter implies that $c_5(x,z)\geq 2$ for all $z \in \G_5(x)$. Using this, the remainder of the distribution follows (as shown in Figure~\ref{fig:dodecahedron}). By Proposition~\ref{prop:dodecahedron}, we obtain the bipartite double scheme of the metric association scheme of the dodecahedron, and we have case (vii). Moreover, the only association scheme with this bipartite double is the scheme of the dodecahedron, which gives case (vi).

(4){\em Eigenvalue $1$.} Next, we consider the case $\theta=\pm 1$; again without loss of generality we assume that $\theta=1$. Recall that we assumed that $c_2=1$. We again consider the configuration of vertices as in the proof of Proposition~\ref{thm: multi 3}, up to $z_8$, and find that $\w_1=\frac13,\w_2=-\frac13,\w_3=\frac13,\w_4=-1,\w_5=1$, and $\w_6=\w_7=\w_8=-\frac13$ (from \eqref{eq: w3}-\eqref{eq: w5} and further). Thus, it is possible that $R_6,R_7,R_8$, and $R_2$ are not distinct. We thus have the partial relation-distribution diagram as in Figure~\ref{fig:theta1basic}, where $p^3_{12}=1$ or $2$.

\begin{figure}[h!]
\begin{center}
\includegraphics[scale=0.65]{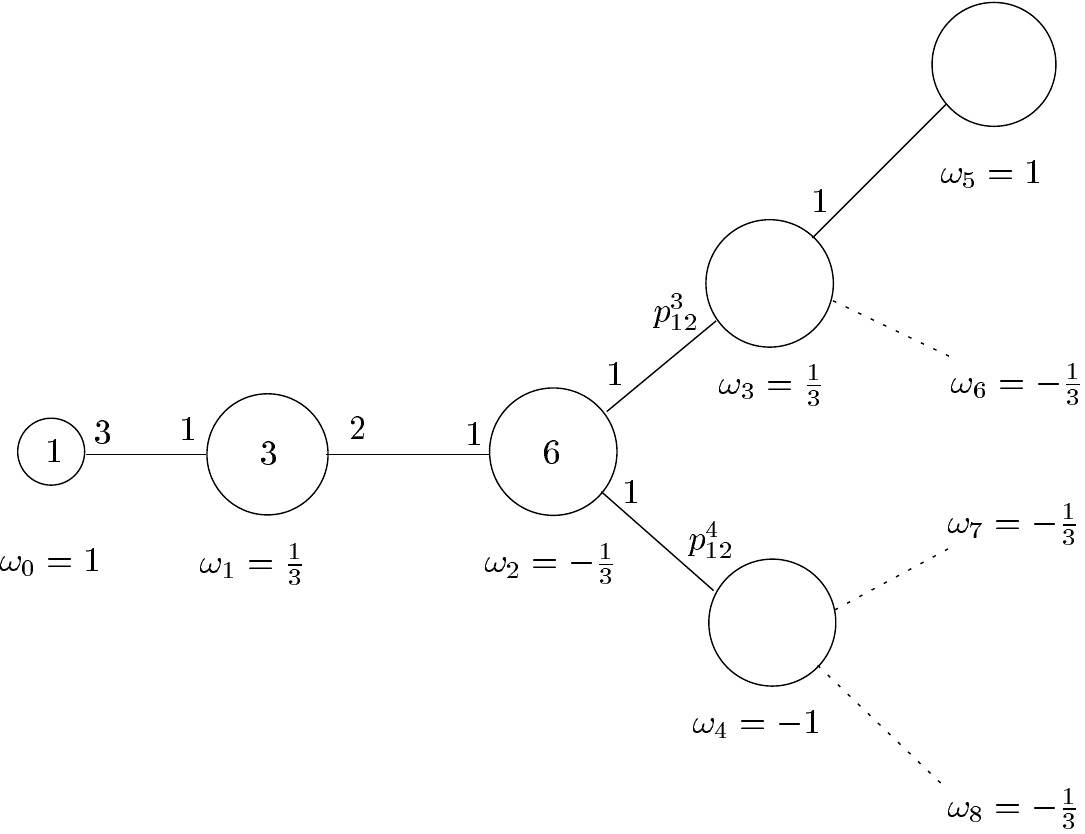}
\caption{Partial relation-distribution diagram for $\theta=1$}
\label{fig:theta1basic}
\end{center}
\end{figure}

(4.1){\em The M\"{o}bius-Kantor graph.} Let us first consider the case that $p^3_{12}=2$, i.e., $R_6=R_2$. Then $k_3=3$.
We claim that in this case, we only have the M\"{o}bius-Kantor graph, with relation-distribution diagram as in Figure~\ref{fig:mk}.

Indeed, if $p^4_{12}=1$, then $R_7$ or $R_8$ should be equal to $R_5$ by Yamazaki's lemma~\ref{lem: cherry}, but $\w_5\neq \w_7$ and $\w_5\neq \w_8$, so we have a contradiction. If $p^4_{12}=2$, then without loss of generality $R_7=R_2$. Now $c_3=2$, hence $c_4(u,v)\geq 2$ for all $u$ and $v$ at distance $4$. In particular, it follows that $p^8_{14}\geq 2$. Because $k_4=3$ and $p^4_{18}=1$, it follows that $p^8_{14}=3$ and $k_8=1$. But now $\theta\w_8=3\w_4$ by \eqref{eq: eigenvector}, and again we have a contradiction. Thus, $p^4_{12}=3$ and hence $k_4=2$. Now $c_3(u,v) \geq 2$ for all $u$ and $v$ at distance $3$, which again implies that $c_4(u,v)\geq 2$ for all $u$ and $v$ at distance $4$. In particular, we obtain that $p^5_{13} \geq 2$, and it then follows that $p^5_{13}=3$ with $k_5=1$. We therefore indeed obtain the relation-distribution diagram of Figure~\ref{fig:mk}. By Proposition~\ref{prop:mk}, we obtain the association scheme of the  M\"{o}bius-Kantor graph and we have case (iii). We also note that the obtained scheme is not the bipartite double of any scheme.

(4.2){\em The Nauru graph.} Next, we consider the case that $p^3_{12}=1$, hence $R_6 \neq R_2$, and $k_3=6$.
By Yamazaki's lemma~\ref{lem: cherry}, we now obtain (without loss of generality) that $R_6=R_7$, and hence that $p^4_{12}\leq 2$. Let us first consider the case that $p^4_{12}=2$. Then $R_8=R_2$, $k_4=3$, which also implies that $k_6=3$, $p^4_{16}=p^6_{14}=1$, and $p^6_{13}=2$. We now claim that in this case, we only have the Nauru graph, with relation-distribution diagram as in Figure~\ref{fig:nauru}.

To show this claim, we observe that it follows from Yamazaki's lemma~\ref{lem: cherry} that $p^5_{13} > 1$. If $p^5_{13}=2$, then $k_5=3$ and there is a relation, $R_9$ say, among the ``distance $5$-relations'', such that $p^5_{19}=1$. Then $\w_9=\theta \w_5-2\w_3=\frac13$ by \eqref{eq: eigenvector}.
Because $c_4(u,v) \geq 2$ for all $u$ and $v$ at distance $4$, it follows that $c_5(u,v)\geq 2$ for all $u$ and $v$ at distance $5$, hence $p^9_{15}=3$ and $k_9=1$. However, now $\theta\w_9=3\w_5$ by \eqref{eq: eigenvector}, which gives a contradiction. Thus, $p^5_{13}=3$, and hence $k_5=2$, and we indeed obtain the relation-distribution diagram as in Figure~\ref{fig:nauru}. By Proposition~\ref{prop:nauru}, we thus obtain the association scheme of the Nauru graph and we have case (iv). Again, we note that this scheme is not the bipartite double of any scheme.

(4.3){\em Girth $8$.} What remains is the case that both $p^3_{12}=1$ and $p^4_{12}=1$. In this case it follows without loss of generality from Yamazaki's lemma~\ref{lem: cherry} that $R_6=R_7$, so that the girth of $\G$ is $8$. We will now first show that the partial relation-distribution diagram is as in Figure~\ref{fig:girth8basic}. Indeed, in this case $z_6$ has at least one neighbor with cosine $\w_3=\frac13$ and at least one neighbor with cosine $\w_4=-1$. By \eqref{eq: eigenvector}, the missing neighbor $z_9$ has cosine $\w_9=\theta \w_6-\w_3-\w_4=\frac13$. Thus, $p^6_{14}=1$, and because $6 \leq p^4_{16}6=p^4_{16}k_4=p^6_{14}k_6=k_6\leq p^6_{13}k_6=p^3_{16}k_3=6$, it follows that $p^4_{16}=1,p^6_{13}=1$, and $k_6=6$. Thus, $R_6 \neq R_8$ and the missing neighbor $z_9$ of $z_6$ is at distance $5$ from $x$, say $(x,z_9) \in R_9$. In the above, we showed that $\w_9=\frac13$.
By applying Lemma~\ref{lemma: multi 3} to the adjacent vertices $z_4$ and $z_8$, and their neighbors, we find that the two remaining neighbors of $z_8$ have cosines $\frac12 (\theta \w_8-\w_4\pm(\w_2-\w_6))=\frac13$. Thus, $p^8_{14}=1$ and $k_8=6$. By Yamazaki's lemma~\ref{lem: cherry}, it now follows that $p^8_{19} \geq 1$. By \eqref{eq: eigenvector}, $z_9$ should now also have a neighbor with cosine $1$, so $p^9_{16}=p^9_{18}=1$, hence $k_9=6$ and $p^8_{19}=1$. The missing neighbor $z_{10}$ of $z_8$ must be at distance $5$ from $x$, say $(x,z_{10})\in R_{10}$. Thus, we find the partial relation-distribution diagram shown in Figure~\ref{fig:girth8basic}. Next, we will distinguish three cases according to the value of $p^5_{13}$.

\begin{figure}
\begin{center}
\includegraphics[scale=0.65]{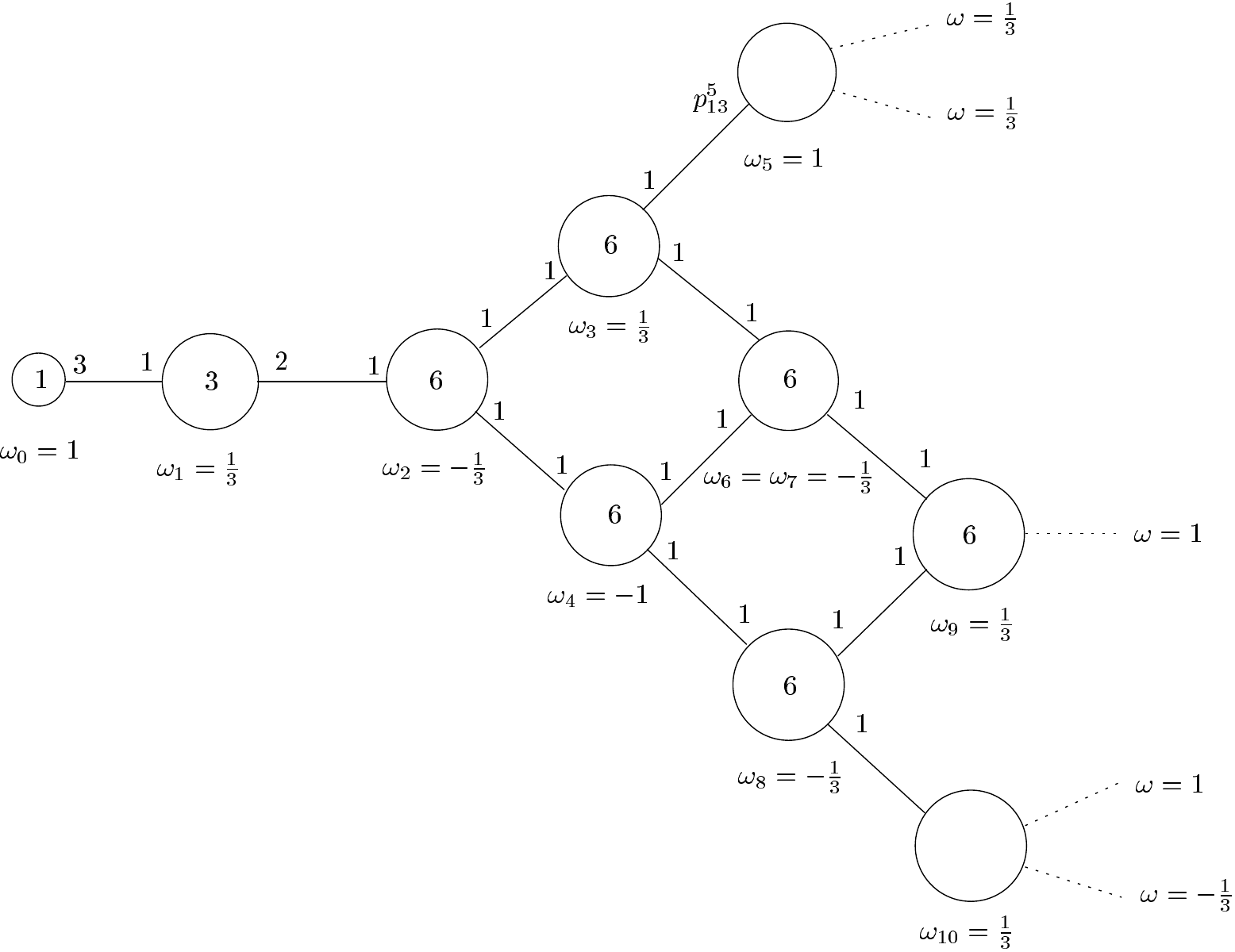}
\caption{Partial relation-distribution diagram for girth $8$}
\label{fig:girth8basic}
\end{center}
\end{figure}

(4.3.1){\em A putative scheme related to the Coxeter graph.} First, if $p^5_{13}=1$, then $k_5=6$. By Yamazaki's lemma~\ref{lem: cherry}, it follows that $p^9_{15}=1$ and hence $p^5_{19}=1$. Let $z_{11}$ be the missing neighbor of $z_5$, with $(x,z_{11})\in R_{11}$, say. Thus, $\w_{11}=\frac13$, but it is clear that $R_{11}\neq R_9$. By Lemma~\ref{lemma: multi 3}, the other two neighbors of $z_{11}$ (i.e., not $z_5$) have cosines $-\frac13$, so $p^{11}_{15}=1$ and $k_{11}=6$. By Yamazaki's lemma~\ref{lem: cherry}, it follows that $R_{10}=R_{11}$, and hence that $p^{10}_{18}=1$. It now easily follows that we obtain the relation-distribution diagram of Figure~\ref{fig:doublecoxeter}. However, by Proposition~\ref{prop:noncoxeter} such a scheme, which is related to the Coxeter graph, does not exist.

(4.3.2){\em A putative $3$-cover of the M\"{o}bius-Kantor graph.} Secondly, if $p^5_{13}=2$, then $k_5=3$, and it follows that $p^9_{15}=p^5_{19}=0$. Let $z_{11}$ be the missing neighbor of $z_9$, with $(x,z_{11})\in R_{11}$, say. Now $\w_{11}=1$. If moreover $p^5_{1,10}=1$, then $k_{10}=3$, $p^{10}_{18}=2$, and $p^{10}_{15}=1$, and it now easily follows that we obtain the relation-distribution diagram of Figure~\ref{fig:3covermk}. However, by Proposition~\ref{prop:non3covermk} such a scheme, whose scheme graph is a putative $3$-cover of the M\"{o}bius-Kantor graph, does not exist. If however $p^5_{1,10}=0$, then we claim that we obtain the partial relation-distribution diagram of Figure~\ref{fig: theta=1-6}.
\begin{figure}
\begin{center}
\includegraphics[scale=0.60]{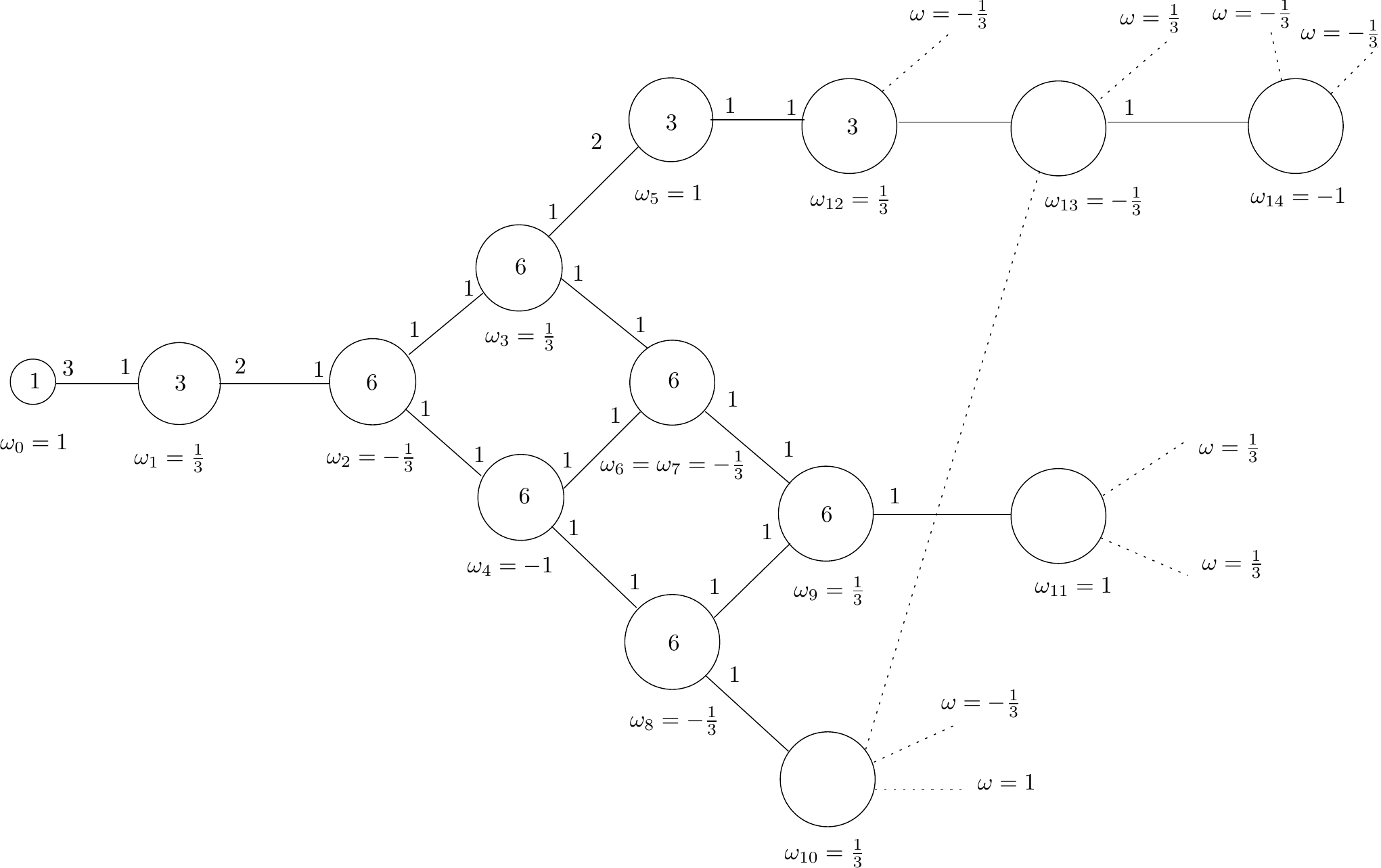}
\caption{Partial relation-distribution diagram for the case $p^5_{13}=2$ and $p^5_{1,10}=0$}
\label{fig: theta=1-6}
\end{center}
\end{figure}
Now, let $z_{12}$ be the missing neighbor of $z_5$, with $(x,z_{12})\in R_{12}$, say. Thus, $\w_{12}=\frac13$. Again, by Lemma~\ref{lemma: multi 3}, the other two neighbors of $z_{12}$ have cosines $-\frac13$, so $p^{12}_{15}=1$ and $k_{12}=3$. Now we easily obtain a relation $R_{13}$ with $\w_{13}=-\frac13$, which is among the distance-$6$ relations, and a relation $R_{14}$ with $\w_{14}=-1$, among the distance-$7$ relations, as in Figure~\ref{fig: theta=1-6}. Let $z_5z_{12}z_{13}z_{14}$ be a path of length $3$, with $(x,z_{i})\in R_{i}$ for $i=5,12,13,14$. Because the girth of $\G$ is $8$, this is the unique path between $z_5$ and $z_{14}$, and $(z_{14},z_5)\in R_3$ or $R_4$. From the partial relation-distribution diagram it follows that $z_5$ has a neighbor $z$ such that $(z_{14},z)\in R_6$. This implies that there are two paths of length $4$ from $z$ to $z_{14}$, one of them being $zz_5z_{12}z_{13}z_{14}$. Because $p^5_{13}=2$, it also follows that $(x,z)\in R_3$. This implies that there must be a path $zv_6v_9v_{11}z_{14}$, with $(x,v_i)\in R_i$ for $i=6,9,11$. This implies that $p^{11}_{1,14}=1$, but this is impossible by Lemma~\ref{lemma: multi 3}. Thus, there is no association scheme with partial relation-distribution diagram of Figure~\ref{fig: theta=1-6}.

(4.3.3){\em The Foster graph} F048A.
Thirdly, and finally, if $p^5_{13}=3$, then $k_5=2$. Again, let $z_{11}$ be the missing neighbor of $z_9$, with $(x,z_{11})\in R_{11}$, say, then $\w_{11}=1$. Let $v_5$ be a vertex at distance $4$ from $z_{11}$ such that $(x,v_5)\in R_5$. Because $\w_5=\w_{11}=1$, it follows that $\hat{v_5}=\hat{x}=\hat{z_{11}}$ (see the proof of Lemma~\ref{lemma: multi 3}), hence $\w_{v_5z_{11}}=1$ and $(v_5,z_{11})\in R_5$. This implies that there are three disjoint paths of length $4$ between $z_{11}$ and $v_5$, and hence $p^{11}_{19}=3$ and $k_{11}=2$. Now $p^{10}_{18}=1$ is impossible by Yamazaki's lemma~\ref{lem: cherry}, hence $p^{10}_{18}=2$ and we obtain the relation-distribution diagram of Figure~\ref{fig:foster48}. By Proposition~\ref{prop:foster48}, it follows that the association scheme is the one with scheme graph the Foster graph F048A, and we obtain case (v). Finally, we observe that this scheme is not the bipartite double of any scheme.
\end{proof}

\section{Partially metric association schemes with a multiplicity three}\label{sec:2pmschemes}
Now we can finally give our main result, the classification of partially metric association schemes with a multiplicity three.

\begin{theorem}\label{thm: multi} Let $(X,\mathcal{R})$ be a partially metric association scheme with rank $d+1$ and a multiplicity three, and let $\G$ be the corresponding scheme graph. Then one of the following holds:
\begin{enumerate}[{\em (i)}]
  \item $d=1$ and $\G$ is the tetrahedron (the complete graph on $4$ vertices),
  \item $d=3$ and $\Gamma$ is the cube,
  \item $d=5$ and $\Gamma$ is the M\"{o}bius-Kantor graph,
  \item $d=6$ and $\Gamma$ is the Nauru graph,
  \item $d=11$ and $\Gamma$ is the Foster graph {\em F048A},
  \item $d=5$ and $\Gamma$ is the dodecahedron,
  \item $d=11$ and $\Gamma$ is the bipartite double of the dodecahedron,
  \item $d=3$ and $\Gamma$ is the icosahedron,
  \item $d=2$ and $\Gamma$ is the octahedron,
  \item $d=2$ and $\Gamma$ is a regular complete $4$-partite graph.
\end{enumerate}
Moreover, the association scheme $(X,\mathcal{R})$ is uniquely determined by $\Gamma$. In all cases, except {\em (vii)}, this is the association scheme that is generated by $\G$. In case {\em (vii)}, the association scheme is the bipartite double scheme of the association scheme of case {\em (vi)}.

\end{theorem}

\begin{proof}
 If $\Gamma$ is complete multipartite, then $(X,\mathcal{R})$ has rank three, and we can easily see that $\G$ is the octahedron ($K_{2,2,2}$) or a regular complete $4$-partite graph, and we obtain cases (ix) and (x). Now, let us assume that $\Gamma$ is not complete multipartite, with valency $k$.

Let $A$ be the adjacency matrix of $\G$ and let $E$ be the minimal scheme idempotent with multiplicity three for corresponding eigenvalue $\theta$. Because $\Gamma$ is not complete multipartite, Theorem~\ref{thm: godsil} implies that $k \leq 5$. Note also that $a_1<k-1$ because $\G$ is not complete.

If $k=2$, then $\G$ is a cycle, but then the corresponding scheme does not have a multiplicity three. Thus, $k>2$. If $k=3$, then we have one of the cases (i)-(vii) by Theorem~\ref{thm: classification}.
If $k>3$, then $a_1>0$ by \cite[Lemma 6.7]{CDKP}.

We first assume that $k=4$. Then $a_1$ is either $1$ or $2$. If $a_1=1$, then $\Gamma$ is locally a disjoint union of two edges and $b_1=2$. Because $-1-\frac{b_1}{\theta+1}$ is an eigenvalue of every local graph of $\Gamma$ by \cite[Prop.~5.2]{CDKP}, it follows that $\theta=-2$. Now equality holds in \eqref{eq: light tail 1}, and hence $E$ is a light tail. If $\eta$ is the corresponding eigenvalue on the associated matrix $F$ for the light tail $E$, then it follows from \eqref{eq: etatheta} that $\eta=-\frac{1}{2}$.  But this is impossible because every eigenvalue of $\G$ must be an algebraic integer.

If $a_1=2$, then $\G$ is locally a quadrangle and hence it is the octahedron. The octahedron is a complete multipartite graph however, which we excluded in this part of the proof (still it occurs as case (ix), of course).

Finally, we assume that $k=5$. Then $a_1=2$ because $ka_1$ must be even. So $\Gamma$ is locally a pentagon and this shows that $\Gamma$ is the icosahedron (see \cite[Prop.~1.1.4]{bcn89}), which is a distance-regular graph with spectrum $\{5^1,\sqrt{5}^3, -1^5, -\sqrt{5}^3\}$. Theorem~\ref{thm: godsil} implies that every minimal scheme idempotent of $(X,\mathcal{R})$ has multiplicity at least three for corresponding eigenvalue $\theta$ if $\theta \neq \pm k$. This implies that we cannot split the idempotent with multiplicity $5$, which shows that the association scheme $(X,\mathcal{R})$ is also uniquely determined by $\Gamma$ in this final case (viii).
\end{proof}

We note that the bipartite double schemes of the (metric) association schemes of the icosahedron, the octahedron, and the regular complete $4$-partite graphs also have a multiplicity three, but these are not partially metric. Analogous to the case of the dodecahedron, the bipartite double scheme of the icosahedron does have a fusion scheme that is partially metric, but this fusion scheme does not have a multiplicity three. The bipartite double graph of the icosahedron is the incidence graph of a group divisible design with the dual property, see \cite{QDK}. Among the $2$-walk-regular graphs with fixed valency, these have a relatively small number of vertices, see \cite{QKP}. The bipartite double scheme of a regular complete $4$-partite graph is a cover of the cube in the sense that it is imprimitive with the association scheme of the cube as a quotient scheme.

\section{Non-commutative association schemes from covers of the cube}\label{sec:cubecovers}

In this section, we present an infinite family of arc-transitive covers of the cube with an eigenvalue with
multiplicity three. By a similar result as Lemma~\ref{prop: cubic 1-walk} (see \cite[Prop.~3.6]{CDKP}), this provides an infinite family of $2$-walk-regular graphs with a multiplicity three, as we already announced in \cite[p.~2705]{CDKP}.
Moreover, by considering the orbitals of the corresponding automorphism groups, we obtain an
infinite family of non-commutative association schemes with a symmetric relation having a multiplicity three (note that
we are careful not to call this a multiplicity of the scheme). This indicates that the restriction to symmetric association schemes in the earlier sections is not without good reason.

Feng, Kwak, and Wang \cite{FK}, \cite[Ex.~3.1]{FKW} constructed covers of the cube from voltage graphs. We will describe
(and generalize somewhat) these covers $\G$ by their incidence matrix $N$ as follows. Let $n$ and $k\leq n-1$ be such
that $k^2+k+1$ is a multiple of $n$.
Let $C$ be the $n\times n$ permutation matrix corresponding to a cyclic permutation of order $n$. Then we let
$$N =\begin{bmatrix}I & I & I & 0\\I & C & 0 & I \\I & 0 & C^{k+1} & C^k \\0 & I & C^k & C^k\end{bmatrix}.$$

\begin{proposition}\label{prop:cubecovers} Let $n$ and $k\leq n-1$ be such
that $k^2+k+1$ is a multiple of $n$. Then the bipartite graph $\G$ with bipartite incidence matrix $N$ is
arc-transitive and it has eigenvalues $\pm 1$ with multiplicity three.
\end{proposition}

\begin{proof}
The arc-transitivity was essentially shown by Feng and Kwak \cite{FK} by using the concept of voltage graphs. The idea
is that the arc-transitivity of the cube can be ``lifted'' to ``transitivity of the nonzero blocks in the matrix $N$'',
which can be combined with using the cyclic group within the blocks. Note that here it is important that both $k$ and
$k+1$ have no common divisors with $n$, that is, $C^k$ and $C^{k+1}$ also represent cyclic permutations of order $n$.

For the multiplicity result, we note that it is not hard to show (see \cite{Hthesis}) that $\G$ has both
eigenvalues $\pm 1$ with multiplicity three if and only if $NN^{\top}$ has eigenvalue $1$ with multiplicity three. We thus
would like to know the nullity of the matrix
$$NN^{\top}-I=\begin{bmatrix}2I & I+C^{-1} & I+C^{-k-1} & I+C^{-k}\\I+C & 2I & I+C^{-k} & C+C^{-k} \\I+C^{k+1} & I+C^{k} & 2I & I+C \\I+C^{k} & C^{-1}+C^{k} & I+C^{-1} & 2I\end{bmatrix}.$$
Using elimination and decomposition, we found that
$$NN^{\top}-I=\frac12 M\begin{bmatrix}I & 0 & 0 & 0\\0 & I & 0 & 0 \\0 & 0 & 0 & -2I \\0 & 0 & -2I & 0\end{bmatrix}M^{\top},$$

where $$M=
\begin{bmatrix}2I & 0 & 0 & 0\\I+C & I-C & I-C & 0 \\I+C^{k+1} & I-C^{k+1} & 0 & 0 \\ I+C^k & I-C^k & 0 & I-C^k\end{bmatrix},$$
which indeed implies that $\nul (NN^{\top}-I)= \nul M =\nul (I-C) +\nul (I-C^k)+ \nul (I-C^{k+1})=3$.  \end{proof}

The cases $(n,k)=(1,0)$ and $(n,k)=(3,1)$ give rise to the cube and the Nauru graph, respectively. By Theorem~\ref{thm: classification}, all other examples give rise to non-symmetric schemes, and hence to non-commutative schemes. Indeed, if the scheme were commutative and non-symmetric, then we could consider its symmetrized scheme. Thus, we may conclude that there exists an infinite family of non-commutative association schemes with a connected and symmetric cubic relation having an eigenvalue with multiplicity three.\\

\noindent {\bf Acknowledgements.}
The authors thank Marc C\'amara for his contribution in the early start of this project and doing some supporting computations. Jack H. Koolen is partially supported by the National Natural Science Foundation of China (no. 11471009 and 11671376).

\end{document}